\documentclass[12 pt]{amsart}
\usepackage{amssymb,latexsym,amsmath,amscd,amsthm,graphicx, color}
\usepackage[all]{xy}
\usepackage{pgf,tikz}
\usepackage{mathrsfs}
\usepackage{cite}
\usetikzlibrary{arrows}
\definecolor{uuuuuu}{rgb}{0.26666666666666666,0.26666666666666666,0.26666666666666666}
\definecolor{xdxdff}{rgb}{0.49019607843137253,0.49019607843137253,1.}
\definecolor{ffqqqq}{rgb}{1.,0.,0.}

\newcommand{\degre}{\ensuremath{^\circ}}
\definecolor{uuuuuu}{rgb}{0.26666666666666666,0.26666666666666666,0.26666666666666666}
\definecolor{qqwuqq}{rgb}{0.,0.39215686274509803,0.}
\definecolor{zzttqq}{rgb}{0.6,0.2,0.}
\definecolor{xdxdff}{rgb}{0.49019607843137253,0.49019607843137253,1.}
\definecolor{qqqqff}{rgb}{0.,0.,1.}
\definecolor{cqcqcq}{rgb}{0.7529411764705882,0.7529411764705882,0.7529411764705882}

\theoremstyle{plain}

\newtheorem{conj}[subsection]{Conjecture}

\newtheorem{theorem}[subsubsection]{Theorem}

\newtheorem{lemma}[subsection]{Lemma}

\theoremstyle{definition}
\newtheorem{prop}[subsection]{Proposition}
\newtheorem{cor}[subsubsection]{Corollary}

\newtheorem{remark}[subsection]{Remark}

\newtheorem{note}[subsection]{Note}

\newcommand{\set}[1]{\{#1\}}

\newcommand{\ga}{\alpha}
\newcommand{\gb}{\beta}

\newcommand{\gl}{\lambda}

\newcommand{\gq}{\theta}

\newcommand{\tbf}{\textbf}
\newcommand{\tit}{\textit}

\newcommand{\D}[1]{\mathbb{#1}}

\newcommand{\te}{\text}

\newcommand{\nd}{\noindent}

\newcommand{\tri}{\triangle}
\begin{document}
To appear, Real Analysis Exchange
\title{Optimal quantizers for some absolutely continuous probability measures}

\author{ Mrinal Kanti Roychowdhury}
\address{School of Mathematical and Statistical Sciences\\
University of Texas Rio Grande Valley\\
1201 West University Drive\\
Edinburg, TX 78539-2999, USA.}
\email{mrinal.roychowdhury@utrgv.edu}

\thanks{The research of the author was supported by U.S. National Security Agency (NSA) Grant H98230-14-1-0320}

\date{}
\maketitle

\pagestyle{myheadings}\markboth{Mrinal Kanti Roychowdhury}{Optimal quantizers for some absolutely continuous probability measures}

\begin{abstract}
The representation of a given quantity with less information is often referred to as `quantization' and it is an important subject in information theory. In this paper, we have considered absolutely continuous probability measures on unit discs, squares, and the real line. For these probability measures the optimal sets of $n$-means and the $n$th quantization errors are calculated for some positive integers $n$.
\end{abstract}

\section{Introduction}

Quantization is a nonlinear, zero-memory operation of converting a continuous signal
into a discrete signal that assumes only a finite number of levels. Quantization
occurs whenever physical quantities are represented numerically. W.F. Sheppard is the first person who studied a system of quantization (see \cite{S}). It has broad applications in signal processing, telecommunications, data compression, image processing and cluster analysis. For some details and comprehensive lists of references one can see \cite{GG, GKL, GN, Z}. Rigorous mathematical treatment of the quantization theory is given in Graf-Luschgy's book (see \cite{GL}).

The quantization of a probability distribution refers to the idea of estimating a given probability by a discrete probability with a given number $n$ of supporting points. Let $P$ denote a Borel probability measure on $\D R^d$ and let $\|\cdot\|$ denote the Euclidean norm on $\D R^d$ for any $d\geq 1$.  The $n$th \textit{quantization
error} for $P$ (of order $2$) is defined by
\begin{equation*} \label{eq1} V_n:=V_n(P)=\inf \Big\{\int \min_{a\in\alpha} \|x-a\|^2 dP(x) : \alpha \subset \mathbb R^d, \text{ card}(\alpha) \leq n \Big\},\end{equation*}
where the infimum is taken over all subsets $\alpha$ of $\mathbb R^d$ with card$(\alpha)\leq n$ for $n\geq 1$. If $\int \| x\|^2 dP(x)<\infty$, then there is some set $\alpha$ for
which the infimum is achieved (see \cite{GL}). A set $\alpha$ for which the infimum is achieved, i.e.,
$V_n=\int \min_{a \in \ga} \|x-a\|^2 dP(x),$ is called an \textit{optimal set of $n$-means}. Elements of an optimal set of $n$-means are called \tit{optimal quantizers}. In some literature it is also refereed to as \tit{principal points} (see \cite{MKT}, and the references therein). For a finite set $\ga$, the error $\int \min_{a \in \ga} \|x-a\|^2 dP(x)$ is often referred to as the \tit{cost} or \tit{distortion error} for $\ga$, and is denoted by $V(P; \ga)$. Thus, $V_n:=V_n(P)=\inf\set{V(P; \ga) :\alpha \subset \mathbb R^d, \text{ card}(\alpha) \leq n}$. It is known that for a continuous probability measure $P$ an optimal set of $n$-means always has exactly $n$ elements (see \cite{GL}). In this paper, we consider both uniform and nonuniform continuous probability distributions, i.e., the probability measures $P$ considered in this paper are absolutely continuous with respect to the Lebesgue measure $\gl$, i.e., for any Borel subset $B$ of $\D R^d$, we have $P(B)=\int_B f(x) d\gl(x)$, where $f$ is the density function, known as Radon-Nikodym derivative of $P$ with respect to $\gl$.
For a finite subset $\ga$ of $\D R^d$ the \tit{Voronoi region} generated by an element $a\in \ga$ is the set of all elements in $\D R^d$ which are nearest $a$, and is denoted by $M(a|\ga)$, i.e.,
\[M(a|\ga)=\set{x \in \D R^d : \|x-a\|=\min_{b \in \ga}\|x-b\|}.\]
The set $\set{M(a|\ga) : a \in \ga}$ is called the \tit{Voronoi diagram} or \tit{Voronoi tessellation} of $\D R^d$ with respect to the set $\ga$. Let us now state the following proposition  (see \cite[Section~4.1]{GL} and \cite[Chapter~6 and Chapter~11]{GG}).
\begin{prop} \label{prop10}
Let $\ga$ be an optimal set of $n$-means for a Borel probability measure $P$ on $\D R^d$. Let $a \in \ga$, and $M (a|\ga)$ be the Voronoi region generated by $a\in \ga$.
Then, for every $a \in\alpha$, $(i)$ $P(M(a|\ga))>0$, $(ii)$ $ P(\partial M(a|\ga))=0$, $(iii)$ $a=E(X : X \in M(a|\ga))$, and $(iv)$ $P$-almost surely the set $\set{M(a|\ga) : a \in \ga}$ forms a Voronoi partition of $\D R^d$.
\end{prop}

\begin{remark}\label{rem1} With respect to a uniform distribution defined on a region with uniform density, the points in an optimal set are the mass centers, also known as centroids, of their own Voronoi regions.
\end{remark}

In \cite{DR}, Dettmann and Roychowdhury considered a uniform distribution on an equilateral triangle, and investigated the optimal sets of $n$-means and the $n$th quantization errors for this distribution for all $n\geq 1$. They first showed that the  Voronoi regions generated by the two points in an optimal set of two-means partition the equilateral triangle into an isosceles trapezoid and an equilateral triangle in the Golden ratio. Then, by mathematical calculation they determined the optimal sets of three- and four-means.  As the number of points increases, so does the number of algebraic equations to be solved. So, they applied a numerical search algorithm that makes random shifts to the point locations, accepting better configurations, and gradually decreasing the shift amplitude in the absence of improvement.  They presented the results of this numerical search
for $n\leq 21$ points and gave conjectures about the optimal configurations for $n$ points, a bound on the quantization errors for $n\to\infty$, and a final conjecture about uniform distributions in more general geometries.

In this paper, due to calculation simplicity, in Section~1 we determine the optimal sets of $n$-means  and the $n$th quantization errors for all $1\leq n\leq 6$ for a uniform distribution on a unit disc.  In Section~2, we determine them for a uniform distribution on a unit square. In Section~3, we determine the optimal sets of $n$-means and the $n$th quantization errors for all $1\leq n\leq 4$ for a nonuniform distribution on $\D R$ supported by the closed interval $[0, 1]$. The technique of this paper can be used to investigate the optimal quantization for many other uniform and nonuniform probability distributions defined on a region.

\section{Optimal sets and quantization error for uniform distribution on a disc}

In this section, we first give some basic results relating to optimal sets and a uniform distribution defined on a unit disc with uniform density. Then, we determine the optimal sets of $n$-means and the $n$th quantization errors with respect to the uniform distribution for all $1\leq n\leq 6$.

\subsection{Basic results}
Let $X:=(X_1, X_2)$ be a bivariate continuous random variable with uniform distribution taking values on the unit disc with center at the origin. Then, the probability density function (pdf) $f(x_1, x_2)$  of the random variable $X$ is given by
\[f(x_1, x_2)=\left\{\begin{array}{ccc}
\frac{1}{\pi} & \te{ for }  x_1^2+x_2^2 \leq 1, \\
\ 0  & \te{ otherwise}.
\end{array}\right.
\]
Notice that the pdf satisfies the following two necessary conditions:

$(i)$ $f(x_1, x_2)\geq 0$ for all $(x_1, x_2) \in \D R^2$, \te{ and }

$(ii)$ $\iint_{\D R^2} f(x_1, x_2)\, dx_1 dx_2=\int_{r=0}^1 \int_{\gq=0}^{2 \pi}f(r \cos \gq, r\sin \gq)\, r dr d\gq =1.$

Let $f_1(x_1)$ and $f_2(x_2)$ represent the marginal pdfs of the random variables $X_1$ and $X_2$, respectively. Then, following the definitions in Probability Theory, we have
\[f_1(x_1)=\int_{-\infty}^\infty f(x_1, x_2) \,dx_2 \te{ and } f_2(x_2)=\int_{-\infty}^\infty f(x_1, x_2) \,dx_1.\]
Since for $-1\leq  x_1\leq 1$,
\[\int_{-\sqrt{1-x_1^2}}^{\sqrt{1-x_1^2}} f(x_1, x_2)\, dx_2=\frac{2 \sqrt{1-x_1^2}}{\pi },\]
we have
\[f_1(x_1)=\left\{\begin{array}{cc} \frac{2 \sqrt{1-x_1^2}}{\pi } & \te{ for } -1\leq x_1\leq 1,\\
\
0   & \te{ otherwise}.
\end{array}\right.
\]
Similarly, we can write
\[f_2(x_2)=\left\{\begin{array}{cc} \frac{2 \sqrt{1-x_2^2}}{\pi } & \te{ for } -1\leq x_2\leq 1,\\
\
0   & \te{ otherwise}.
\end{array}\right.
\]
Notice that both $f_1(x_1)$ and $f_2(x_2)$ satisfy the necessary conditions for pdfs: $f_1(x_1)\geq 0$, $f_2(x_2)\geq 0$ for all $x_1, x_2\in \D R$, and
\[\int_{-\infty}^\infty f_1(x_1)\,dx_1=1=\int_{-\infty}^\infty f_2(x_2)\,dx_2.\]
For the random variable $X$, let $E(X)$  and $V(X)$ represent the expected vector and the expected squared distance of $X$. On the other hand, for $i=1, 2$, by $E(X_i)$ and $V(X_i)$ we denote the expectation and the variance of $X_i$, respectively. Let $i$ and $j$ be the unit vectors in the positive directions of $x_1$- and $x_2$-axes, respectively. By the position vector $\tilde a$ of a point $A$, it is meant that $\overrightarrow{OA}=\tilde a$. In the sequel, we will identify the position vector of a point $(a_1, a_2)$ by $(a_1, a_2):=a_1 i +a_2 j$, and apologize for any abuse in notation. For any two vectors $\vec u$ and $\vec v$, let $\vec u \cdot \vec v$ denote the dot product between the two vectors $\vec u$ and $\vec v$. Then, for any vector $\vec v$, by $(\vec v)^2$, we mean $(\vec v)^2:= \vec v\cdot \vec v$. Thus, $|\vec v|:=\sqrt{\vec v\cdot \vec v}$, which is called the length of the vector $\vec v$. For any two position vectors $\tilde a:=( a_1, a_2)$ and $\tilde b:=( b_1, b_2)$, we write $\rho(\tilde a, \tilde b):=(( a_1-b_1, a_2-b_2))^2=(a_1-b_1)^2 +(a_2-b_2)^2$.

Let us now prove the following lemma.
\begin{lemma}
Let $X:=(X_1, X_2)$ be a bivariate continuous random variable with uniform distribution taking values on the unit disc with center $(0, 0)$. Then,
\[E(X)=(0, 0) \te{ and } V(X)=V(X_1)+V(X_2)=\frac 1 {2}.\]
\end{lemma}
\begin{proof}
We have
\begin{align*}
E(X_1)& =\int_{-\infty}^\infty x_1 f_1(x_1)\, dx_1=\int_{-1}^1 \frac{x_1 (2 \sqrt{1-x_1^2})}{\pi } \, dx_1=0,\\
E(X_2) &=\int_{-\infty}^\infty x_2 f_2(x_2)\, dx_2=\int_{-1}^1 \frac{x_2(2 \sqrt{1-x_2^2})}{\pi } \, dx_2=0,
\end{align*}
and so,
\[E(X)=\iint(x_1 i+x_2 j) f(x_1, x_2) dx_1dx_2 =i \int x_1 f_1(x_1) dx_1+j \int x_2 f_2(x_2) dx_2=0 i +0 j=(0, 0).\]
Again, \begin{align*}
E(X_1^2)& =\int_{-\infty}^\infty x_1^2 f_1(x_1)\, dx_1=\int_{-1}^1 \frac{x_1^2 (2 \sqrt{1-x_1^2})}{\pi } \, dx_1=\frac 1 4,\\
E(X_2^2)&=\int_{-\infty}^\infty x_2^2 f_2(x_2)\, dx_2=\int_{-1}^1 \frac{x_2^2 (2 \sqrt{1-x_2^2})}{\pi } \, dx_2=\frac 1 4,
\end{align*}
and so,
\[V(X_1)=E(X_1^2)-[E(X_1)]^2=\frac{1}{4} \te{ and } V(X_2)=E(X_2^2)-[E(X_2)]^2=\frac{1}{4}. \]
Thus, we have
\[V(X)=E\|X-E(X)\|^2=\iint \Big((x_1-E(X_1))^2 +(x_2-E(X_2))^2\Big)f(x_1, x_2)\, dx_1 dx_2,\]
implying
\[V(X)=\int (x_1-E(X_1))^2f_1(x_1)\,dx_1+\int (x_2-E(X_2))^2f_2(x_2)\,dx_2=V(X_1)+V(X_2)=\frac{1}{2}.\]
Hence, the lemma is yielded.
\end{proof}
\begin{cor} \label{cor1} Since $E(X_1)=0$ and $E(X_2)=0$, for any two real numbers $a$ and $b$, we have $ E(X_1-a)^2 =E(X_1^2)+a^2=V(X_1)+a^2$, and similarly $E(X_2-b)^2=V(X_2)+b^2$.
Thus, for any $(a, b) \in \D R^2$, we have \begin{align*} &E\|X-(a, b)\|^2\\
&=\iint_{\D R^2} [(x_1-a)^2+(x_2-b)^2]f(x_1, x_2)dx_1dx_2=\int_{\D R} (x_1-a)^2 f_1(x_1) dx_1+\int_{\D R} (x_2-b)^2f_2(x_2) dx_2\\
&=E(X_1-a)^2 +E(X_2-b)^2=V(X_1)+V(X_2)+a^2+b^2=\frac 1 {2}+a^2+b^2.
\end{align*}

\end{cor}
\begin{note}
From Corollary~\ref{cor1} it is clear that the optimal set of one-mean consists of the expected vector $(0, 0)$ of the random variable $X$, which is the center of the disc, and the corresponding quantization error is  $\frac 1 {2}$ which is the expected squared distance of the random variable $X$.
\end{note}

Let us now give the following lemma.

\begin{lemma} \label{lemma12}
For the uniform distribution on the unit disc with center at the origin let $g(\gq_1, \gq_2)$ be the position vector of the centroid of the sector which makes a central angle of $(\gq_2-\gq_1)$ radians, and let $V(\gq_1, \gq_2)$ be the distortion error of the sector with respect to the centroid. Then,
\begin{align*}
g(\gq_1, \gq_2)&=\Big(\frac{2 (\sin \theta_1-\sin\theta_2)}{3 (\theta _1-\theta _2)},-\frac{2 (\cos \theta _1-\cos \theta _2)}{3 (\theta _1-\theta _2)}\Big) \te{ and } \\
V(\gq_1, \gq_2)&=\frac{9 \left(\theta _2-\theta _1\right){}^2-32 \sin ^2\left(\frac{1}{2} \left(\theta _2-\theta _1\right)\right)}{36 \pi  \left(\theta _2-\theta _1\right)}.
\end{align*}
\end{lemma}

\begin{proof} Using the definitions of centroid and the distortion error, we have
\begin{align*}
g(\gq_1, \gq_2)&=\frac{\int _0^1\int _{\theta _1}^{\theta _2}\frac{r}{\pi } (r \cos (\theta ), \, r \sin (\theta ))d\theta dr}{\int _0^1\int _{\theta _1}^{\theta _2}\frac{r}{\pi }d\theta dr}=\Big(\frac{2 (\sin \theta_1-\sin\theta_2)}{3 (\theta _1-\theta _2)},-\frac{2 (\cos \theta _1-\cos \theta _2)}{3 (\theta _1-\theta _2)}\Big), \te{ and }
\end{align*}
\begin{align*}
&V(\gq_1, \gq_2)=\int _0^1\int _{\theta _1}^{\theta _2}\frac{r \left(\Big(\frac{2 \left(\cos \theta _1-\cos\theta _2\right)}{3 \left(\theta _1-\theta _2\right)}+r \sin (\theta )\Big){}^2+\left(r \cos \theta-\frac{2 \left(\sin \theta _1-\sin \theta _2\right)}{3 \left(\theta _1-\theta _2\right)}\right){}^2\right)}{\pi }d\theta dr\\
&=\int_{\theta _1}^{\theta _2} \frac{\frac{16 \left(\sin \left(\theta -\theta _1\right)-\sin \left(\theta -\theta _2\right)\right)}{\theta _1-\theta _2}-\frac{16 \left(\cos \left(\theta _1-\theta _2\right)-1\right)}{\left(\theta _1-\theta _2\right){}^2}+9}{36 \pi } \, d\theta=\frac{9 \left(\theta _2-\theta _1\right){}^2-32 \sin ^2\left(\frac{1}{2} \left(\theta _2-\theta _1\right)\right)}{36 \pi  \left(\theta _2-\theta _1\right)}.
\end{align*}
Thus, the lemma is yielded.
\end{proof}

We now give the following note.

\begin{note} \label{note100} With respect to any diagonal the unit disc equipped with the uniform distribution has maximum symmetry, i.e., with respect to any diagonal the unit disc is geometrically symmetric as well as symmetric with respect to the uniform distribution. By the `symmetric with respect to the uniform distribution' it is meant that if two regions of similar geometrical shape lie in the opposite sides of a diagonal and are equidistant from the diagonal, then they have the same probability. This motivates us to give the following conjecture.

\begin{conj} \label{conj1}  $(i)$ The Voronoi regions of the points in an optimal set of two-means with respect to the uniform distribution defined on a disc partition the disc into two regions bounded by the semicircles; $(ii)$ the Voronoi regions of the points in an optimal set of three-means partition the disc into three sectors each subtending a central angle of $\frac{2 \pi}{3}$ radians; $(iii)$ for $n=4, 5, 6$, the Voronoi regions of the points in an optimal set of $n$-means either form a regular $n$-gon with center same as the center of the disc, or one of the points lies at the center of the disc and the other $(n-1)$ points form a regular $(n-1)$-gon with center as the center of the disc.
 \end{conj}
Under the above conjecture, in the following subsections, we determine the optimal sets of $n$-means for $2\leq n\leq 6$. It is extremely difficult and the answer is not known yet how the points in an optimal set of $n$-means for higher values of $n$ are located.

\end{note}
\subsection{Optimal sets of $2$-means} Let $\tilde p$ and $\tilde q$ be the position vectors of the two points in an optimal set of two-means. By Conjecture~\ref{conj1},  it is clear that $\tilde p$ and $\tilde q$ can lie on any diameter and will be  equidistant from the center of the disc. Also, the boundary of the Voronoi regions of $\tilde p$ and $\tilde q$ will be another diameter which is perpendicular to the diameter containing the two points.

First, we assume that the boundary of the Voronoi regions of $\tilde p$ and $\tilde q$ is the diagonal along the $x_2$-axis. Then, $\tilde p$ and $\tilde q$ will be the centroids of the right and left halves of the disc, and so by Lemma~\ref{lemma12}, we have
\begin{align*} \tilde p& =g(-\frac{\pi}2, \frac{\pi}2) =(\frac{4}{3 \pi },0) \te{ and } \tilde q =g(\frac {\pi}2, \frac {3\pi}2)=(-\frac{4}{3 \pi },0).
\end{align*} Notice that because of the uniform distribution the two quantization errors due to the points $\tilde p$ and $\tilde q$ will be the same, and so by Lemma~\ref{lemma12}, the quantization error for an optimal set of two-means is given by
\begin{align*}
V_2=2 \times V(-\frac{\pi}2, \frac{\pi}2)=\frac{9 \pi ^2-32}{18 \pi ^2}=0.319873.
\end{align*}
Similarly, it can be proved that if the boundary of the two Voronoi regions is the diagonal along the $x_1$-axis, then $\tilde p$ and $\tilde q$ will be the centroids of the upper and lower halves of the disc, and so
\[\tilde p=(0, \frac{4}{3 \pi }) \te{ and } \tilde q=(0, -\frac{4}{3 \pi }) \te{ with quantizaiton error } \frac{9 \pi ^2-32}{18 \pi ^2}.\]
Thus, due to rotational symmetry, we can deduce the following theorem (see Figure~\ref{FigA}$(a)$).
\begin{theorem}
For the uniform distribution on the unit disc, there are uncountably many optimal sets of two-means with quantization error $0.319873$. Any such two means form a diameter of the circle $x_1^2+x_2^2=\frac {16}{9 \pi^2}$.
\end{theorem}

\subsection{Optimal sets of 3-means}
By  Conjecture~\ref{conj1}, the Voronoi regions of the three points in an optimal set of three-means partition the disc into three sectors each making a central angle of $\frac {2 \pi}{3}$ radians. In Figure~\ref{FigA}$(b)$, we have considered such a configuration of three points $P$, $Q$, and $R$ with the Voronoi regions sectors $AOB$, $BOC$ and $COA$, respectively. Hence by Lemma~\ref{lemma12},
\begin{align*} \tilde p&=g(0,  \frac {2\pi} 3)=(\frac{\sqrt{3}}{2 \pi },\frac{3}{2 \pi }), \  \tilde q=g(\frac {2\pi} 3, \frac {4\pi} 3)=(-\frac{\sqrt{3}}{\pi },0), \te{ and } \tilde r=g(\frac {4\pi} 3, \frac {6\pi} 3)=(\frac{\sqrt{3}}{2 \pi },-\frac{3}{2 \pi }),
\end{align*}
with quantization error
\begin{align*}
V_3=3 \times (\te{Quantization error due to } P)=3\times V(0,  \frac {2\pi} 3)=0.196036.
\end{align*}
Notice that $P$, $Q$, and $R$ lie on the circle $x_1^2+x_2^2=\frac 3{\pi^2}$. Thus, we deduce the following theorem.

\begin{figure}\begin{tikzpicture}[line cap=round,line join=round,>=triangle 45,x=1.0cm,y=1.0cm]
\clip(-2.2,-3.4) rectangle (2.2,2.2);
\draw [color=ffqqqq] (0.,0.) circle (1.cm);
\draw(0.,0.) circle (1.9804039991880444cm);
\draw (0.3743093222426975,0.9273039044898994)-- (-0.38723248435509183,-0.9219821056073613);
\draw (0.,6.)-- (0.,0.);
\draw (-3.68,0.)-- (4.,0.);
\draw (0.,0.)-- (0.,6.);
\draw (0.,0.)-- (4.,0.);
\draw (0.,0.)-- (-8.,0.);
\draw (0.,0.)-- (0.,-2.);
\begin{scriptsize}
\draw [fill=xdxdff] (0.3743093222426975,0.9273039044898994) circle (1.5pt);
\draw[color=qqqqff] (0.5561800934114638,1.0622878702573413) node {$P$};
\draw [fill=xdxdff] (-0.38723248435509183,-0.9219821056073613) circle (1.5pt);
\draw[color=qqqqff] (-0.5877773286158623,-1.0987667896564022) node {$Q$};
\draw[color=qqqqff] (0.15276774216295861,-0.15880227130029842) node {$O$};
\draw [fill=xdxdff] (0.,6.) circle (1.5pt);
\draw[color=xdxdff] (0.19590755802205334,5.9126910144605755) node {$B$};
\draw [fill=xdxdff] (-3.68,0.) circle (1.5pt);
\draw[color=xdxdff] (-3.489957611731146,0.3977498958073007) node {$C$};
\draw [fill=xdxdff] (4.,0.) circle (1.5pt);
\draw[color=xdxdff] (4.1866187192586,0.3977498958073007) node {$D$};
\draw [fill=xdxdff] (-8.,0.) circle (1.5pt);
\draw[color=xdxdff] (-7.813228036404071,0.3977498958073007) node {$E$};
\draw [fill=xdxdff] (0.,-6.) circle (1.5pt);
\draw[color=xdxdff] (0.8887393568478427,-4.9786248630808165) node {$F$};
\draw [fill=uuuuuu] (0.,0.01835184555454793) circle (1.5pt);
\draw[] (0.04276774216295861,-2.5582947786169992) node { $(a)$ };
\end{scriptsize}
\end{tikzpicture}\qquad
\begin{tikzpicture}[line cap=round,line join=round,>=triangle 45,x=1.0cm,y=1.0cm]
\clip(-2.2,-3.4) rectangle (2.3,2.3);
\draw [shift={(0.,0.)},color=qqwuqq,fill=qqwuqq,fill opacity=0.1] (0,0) -- (0.:0.40519331260804353) arc (0.:120.0003803191083:0.40519331260804353) -- cycle;
\draw(0.,0.) circle (2.0073020397932857cm);
\draw (0.,0.)-- (-1.0036625588930772,1.7383678974355474);
\draw (0.,0.)-- (-1.0054417292685358,-1.7373394625126608);
\draw (0.,0.)-- (2.0073020397932857,0.);
\draw (0.,0.)-- (2.0073020397932857,0.);
\draw [color=ffqqqq] (0.,0.) circle (1.1397458329113863cm);
\draw (0.,0.)-- (-3.430788544548882, 0.);
\draw (0.,0.)-- (0.,2.6929868544624704);
\draw (0.,0.)-- (0.,-2.0);
\draw (0.,0.)-- (4.,0.);
\draw [dotted] (0.6102660268693533,0.9625985352616965)-- (-1.1397458329113863,0.);
\draw [dotted] (-1.1397458329113863,0.)-- (0.6532220947527585,-0.933981508685096);
\draw [dotted] (0.6102660268693533,0.9625985352616965)-- (0.6532220947527585,-0.933981508685096);
\begin{scriptsize}
\draw [fill=uuuuuu] (0.,0.) circle (1.5pt);
\draw [fill=xdxdff] (-1.0036625588930772,1.7383678974355474) circle (1.5pt);
\draw[color=qqqqff] (-1.1346931064366357,1.936626004260788) node {$B$};
\draw [fill=xdxdff] (-1.0054417292685358,-1.7373394625126608) circle (1.5pt);
\draw[color=qqqqff] (-1.0941737751758314,-1.9397233530228342) node {$C$};
\draw [fill=uuuuuu] (2.0073020397932857,0.) circle (1.5pt);
\draw[color=qqqqff] (2.1338662819349152,0.18078831629259673) node {$A$};
\draw[color=qqwuqq] (0.4320543689811326,0.45091719136462616) node {$120\textrm{\degre}$};
\draw [fill=xdxdff] (0.6102660268693533,0.9625985352616965) circle (1.5pt);
\draw[color=qqqqff] (0.7021832440531616,1.1532522665519025) node {$P$};
\draw [fill=xdxdff] (0.6532220947527585,-0.933981508685096) circle (1.5pt);
\draw[color=qqqqff] (0.7832219065747703,-0.9672594027635284) node {$R$};
\draw [fill=uuuuuu] (-1.1397458329113863,0.) circle (1.5pt);
\draw[color=qqqqff] (-1.2832639877262517,-0.18985989004023712) node {$Q$};
\draw[color=qqqqff] (0.15276774216295861,-0.15880227130029842) node {$O$};
\draw [fill=xdxdff] (-3.430788544548882,0.) circle (1.5pt);
\draw[] (-0.02326398772625107,-2.5582947786169992) node { $(b)$ };
\end{scriptsize}
\end{tikzpicture} \qquad
\begin{tikzpicture}[line cap=round,line join=round,>=triangle 45,x=1.0cm,y=1.0cm]
\clip(-2.4,-3.4) rectangle (2.3,2.5);
\draw(0.,0.) circle (1.9940258199265823cm);
\draw [color=ffqqqq] (0.,0.) circle (1.226196159861592cm);
\draw (0.,-2.0) -- (0.,2.4);
\draw [domain=-2.4:2.4] plot(\x,{(-0.-0.*\x)/3.85861419223955});
\draw (0.9160804125828327,0.8150789533176415)-- (-0.9136263155800033,-0.8178288194598079);
\draw (-0.8934272702187612,0.8398480429748989)-- (0.8858861009090672,-0.8477988196945341);
\draw [dotted] (-0.8934272702187612,0.8398480429748989)-- (0.9160804125828327,0.8150789533176415);
\draw [dotted] (0.8858861009090672,-0.8477988196945341)-- (0.9160804125828327,0.8150789533176415);
\draw [dotted] (-0.8934272702187612,0.8398480429748989)-- (-0.9136263155800033,-0.817828819459808);
\draw [dotted] (-0.9136263155800033,-0.817828819459808)-- (0.8858861009090672,-0.8477988196945341);
\begin{scriptsize}
\draw [fill=xdxdff] (0.,3.8587587184996655) circle (1.5pt);
\draw[color=xdxdff] (0.18838882901011528,4.212197473351449) node {$B_1$};
\draw[color=black] (0.2092457114179062,4.504193827060523) node {$a$};
\draw [fill=xdxdff] (-3.85861419223955,0.) circle (1.5pt);
\draw[color=xdxdff] (-3.6701344164312033,0.35367422791012393) node {$C_1$};
\draw[color=black] (-4.546123477558422,0.35367422791012393) node {$b$};
\draw [fill=xdxdff] (0.9160804125828327,0.8150789533176415) circle (1.5pt);
\draw[color=qqqqff] (1.0643778901373335,1.0045219945905982) node {$P$};
\draw [fill=xdxdff] (-0.8934272702187612,0.8398480429748989) circle (1.5pt);
\draw[color=qqqqff] (-1.087308527496294,0.9793807001438524) node {$Q$};
\draw [fill=xdxdff] (0.8858861009090672,-0.8477988196945341) circle (1.5pt);
\draw[color=qqqqff] (0.98095036050617,-1.0228800110040788) node {$S$};
\draw [fill=uuuuuu] (0.,0.) circle (1.5pt);
\draw[color=qqqqff] (0.06324753456336983,-0.1886047146924408) node {$O$};
\draw [fill=uuuuuu] (1.9940258199265823,0.) circle (1.5pt);
\draw[color=qqqqff] (2.148935775342461, 0.18078831629259673) node {$A$};
\draw [fill=uuuuuu] (0.,1.9940258199265826) circle (1.5pt);
\draw[color=qqqqff] (0.14667506419453347,2.203364291834682) node {$B$};
\draw [fill=uuuuuu] (-1.9940258199265823,0.) circle (1.5pt);
\draw[color=qqqqff] (-2.1818664127014216,0.1802466982789601) node {$C$};
\draw [fill=uuuuuu] (0.,-1.9940258199265826) circle (1.5pt);
\draw[color=qqqqff] (0.14667506419453347,-2.1282947786169992) node {$D$};
\draw [fill=uuuuuu] (-0.9136263155800033,-0.817828819459808) circle (1.5pt);
\draw[color=qqqqff] (-0.9587397034183849,-1.0020231285962877) node {$R$};
\draw[] (0.06324753456336983,-2.5582947786169992) node { $(c)$ };
\end{scriptsize}
\end{tikzpicture}
\caption{$(a)$ Optimal quantizers $P$, $Q$ form a diameter of the circle $x_1^2+x_2^2=\frac {16}{9 \pi^2}$; $(b)$ optimal quantizers $P$, $Q$, $R$ form an equilateral triangle inscribed in the circle $x_1^2+x_2^2=\frac {3}{\pi^2}$; $(c)$ optimal quantizers $P$, $Q$, $R$ and $S$ form a square inscribed in the circle $x_1^2+x_2^2=\frac{32}{9 \pi ^2}$.} \label{FigA}
\end{figure}
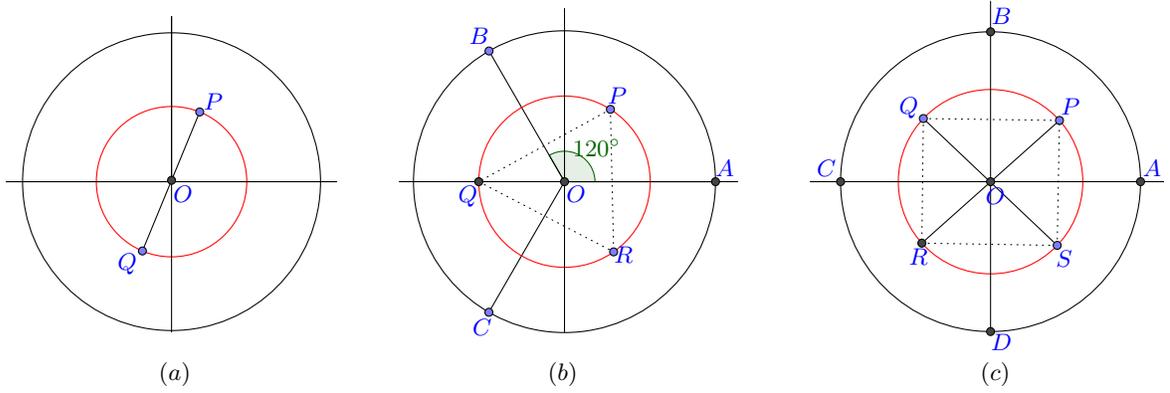

\begin{theorem}
For the uniform distribution on the unit disc, there are uncountably many optimal sets of three-means with quantization error $0.196036$. Any such three means form an equilateral triangle inscribed in the circle $x_1^2+x_2^2=\frac 3{\pi^2}$.
\end{theorem}

\subsection{Optimal sets of $4$-means}
By  Conjecture~\ref{conj1}, the Voronoi regions of the four points in an optimal set of four-means partition the disc into four sectors each making a central angle of $\frac {\pi}{2}$ radians. As shown for $n=5$ in the next subsection, we can show that the optimal set of four-means does not contain the center of the disc.  In Figure~\ref{FigA}$(c)$, we have considered such a configuration of four optimal quantizers $P$, $Q$, $R$, and $S$ with the Voronoi regions sectors $AOB$, $BOC$, $COD$, and $DOA$, respectively. Hence by Lemma~\ref{lemma12},
\begin{align*} \tilde p & =g(0,  \frac {\pi} 2)=(\frac{4}{3 \pi },\frac{4}{3 \pi }), \  \tilde q =g(\frac {\pi} 2, \pi)=(-\frac{4}{3 \pi },\frac{4}{3 \pi }), \
 \tilde r =g(\pi, \frac {3\pi} 2)=(-\frac{4}{3 \pi },-\frac{4}{3 \pi }),\\
  \tilde s &=g(\frac {3\pi} 2, 2 \pi)=(\frac{4}{3 \pi },-\frac{4}{3 \pi }),
\end{align*}
with quantization error
$
V_4=4 \times (\te{Quantization error due to } P)=4\times V(0,  \frac {\pi} 2)=\frac{9 \pi ^2-64}{18 \pi ^2}=0.139747.
$
Thus, we can deduce the following theorem.
\begin{theorem}
For the uniform distribution on the unit disc, there are uncountably many optimal sets of four-means with quantization error $0.139747$. Any such four means form a square inscribed in the circle $x_1^2+x_2^2=\frac{32}{9 \pi ^2}$.
\end{theorem}

\subsection{Optimal sets of $5$-means} By  Conjecture~\ref{conj1}, one of the following two cases can arise:

\tbf{Case 1.} In this case, we assume that the five optimal quantizers are equidistant from the center of the disc. So, the Voronoi regions of the points partition the disc into five sectors each making a central angle of $\frac{2\pi}{5}$ radians. As shown in Figure~\ref{FigB}$(a)$, we assume that the five optimal quantizers are $P$, $Q$, $R$, $S$ and $T$ with Voronoi regions the sectors $AOB$, $BOC$, $COD$, $DOE$ and $EOA$, respectively, each subtending a central angle of $\frac{2\pi}{5}$ radians. Hence by Lemma~\ref{lemma12},
\begin{align*} \tilde p & =g(0,  \frac {2\pi} 5)=(0.504551,0.366578), \ \tilde q=g(\frac {2\pi} 5, \frac {4\pi} 5)=(-0.192721,0.593135), \\
 \tilde r & =(\frac {4\pi} 5, \frac {6\pi} 5)=(-0.62366,0),  \ \tilde s  =(\frac {6\pi} 5, \frac {8\pi} 5)=(-0.192721,-0.593135), \\
  \tilde t& =(\frac {8\pi} 5, \frac {10\pi} 5)=(0.504551,-0.366578),
\end{align*}
with quantization error $
V_5=5 \times (\te{Quantization error due to } P)=5\times V(0,  \frac {2\pi} 5)=5\times 0.0222098=0.111049.$

\begin{figure}\begin{tikzpicture}[line cap=round,line join=round,>=triangle 45,x=1.0cm,y=1.0cm]
\clip(-2.2,-3.4) rectangle (2.3,2.3);
\fill[color=zzttqq,fill=zzttqq,fill opacity=0.1] (1.0091,-0.733156) -- (1.0091,0.733156) -- (-0.3854455825217787,1.186271327055918) -- (-1.2473221513812598,0.) -- (-0.3854455825217793,-1.1862713270559175) -- cycle;
\draw(0.,0.) circle (1.9940258199265823cm);
\draw (0.,-2.2) -- (0.,2.2);
\draw [domain=-2.2:2.2] plot(\x,{(-0.-0.*\x)/3.85861419223955});
\draw [dotted] (0.,0.)-- (0.6161887643200652,1.8964309576832978);
\draw [dotted] (0.,0.)-- (-1.6131996408052047,1.1720605314743078);
\draw [dotted] (0.,0.)-- (-1.6131996408052047,-1.1720605314743078);
\draw [dotted] (0.,0.)-- (0.6161887643200652,-1.8964309576832978);
\draw [color=zzttqq] (1.0091,-0.733156)-- (1.0091,0.733156);
\draw [color=zzttqq] (1.0091,0.733156)-- (-0.3854455825217787,1.186271327055918);
\draw [color=zzttqq] (-0.3854455825217787,1.186271327055918)-- (-1.2473221513812598,0.);
\draw [color=zzttqq] (-1.2473221513812598,0.)-- (-0.3854455825217793,-1.1862713270559175);
\draw [color=zzttqq] (-0.3854455825217793,-1.1862713270559175)-- (1.0091,-0.733156);
\draw [dotted] (0.,0.)-- (1.9940258199265823,0.);
\draw [dotted,color=ffqqqq] (0.,0.) circle (1.2473173334544823cm);
\begin{scriptsize}
\draw [fill=xdxdff] (0.,3.8587587184996655) circle (1.5pt);
\draw[color=xdxdff] (-2.9974547382915024,3.1582764937211794) node {$B_1$};
\draw[color=black] (0.1456335798886154,2.90197315447044) node {$a$};
\draw [fill=xdxdff] (-3.85861419223955,0.) circle (1.5pt);
\draw[color=xdxdff] (-2.9974547382915024,3.1582764937211794) node {$C_1$};
\draw[color=black] (-2.956985789988754,0.2175329170547984) node {$b$};
\draw [fill=uuuuuu] (0.,0.) circle (1.5pt);
\draw[color=qqqqff] (-0.2185869548361193,-0.14668761766993682) node {$O$};
\draw [fill=uuuuuu] (1.9940258199265823,0.) circle (1.5pt);
\draw[color=qqqqff] (2.088143098420534,0.1905536181862995) node {$A$};
\draw [fill=qqqqff] (1.0091,0.733156) circle (1.5pt);
\draw[color=qqqqff] (1.1033986897203252,0.91899468763577) node {$P$};
\draw [fill=qqqqff] (-0.385443,1.18627) circle (1.5pt);
\draw[color=qqqqff] (-0.5693178401266046,1.3101945212290043) node {$Q$};
\draw [fill=xdxdff] (-1.24732,0.) circle (1.5pt);
\draw[color=qqqqff] (-1.3517175073130716,0.1905536181862995) node {$R$};
\draw [fill=qqqqff] (-0.385443,-1.18627) circle (1.5pt);
\draw[color=qqqqff] (-0.4614006446526091,-1.3867974690153898) node {$S$};
\draw [fill=qqqqff] (1.0091,-0.733156) circle (1.5pt);
\draw[color=qqqqff] (1.1438676380230735,-0.5783563995659194) node {$T$};
\draw [fill=xdxdff] (0.6161887643200652,1.8964309576832978) circle (1.5pt);
\draw[color=qqqqff] (0.7121988561270916,2.079104538981223) node {$B$};
\draw [fill=xdxdff] (-1.6131996408052047,1.1720605314743078) circle (1.5pt);
\draw[color=qqqqff] (-1.6889587431693074,1.3371738200975032) node {$C$};
\draw [fill=xdxdff] (-1.6131996408052047,-1.1720605314743078) circle (1.5pt);
\draw[color=qqqqff] (-1.702448392603557,-1.3472664173181381) node {$D$};
\draw [fill=xdxdff] (0.6161887643200652,-1.8964309576832978) circle (1.5pt);
\draw[color=qqqqff] (0.7121988561270916,-2.0947695901621119) node {$E$};
\draw [fill=uuuuuu] (-0.3854455825217787,1.186271327055918) circle (1.5pt);
\draw [fill=uuuuuu] (-1.2473221513812598,0.) circle (1.5pt);
\draw [fill=uuuuuu] (-0.3854455825217793,-1.1862713270559175) circle (1.5pt);
\draw [fill=uuuuuu] (-0.38544856018153495,1.1862672286588642) circle (1.5pt);
\draw[] (-0.0285869548361193 ,-2.5582947786169992) node { $(a)$ };
\end{scriptsize}
\end{tikzpicture}\qquad
\begin{tikzpicture}[line cap=round,line join=round,>=triangle 45,x=1.0cm,y=1.0cm]
\clip(-2.2,-3.4) rectangle (2.3,2.3);
\draw(0.,0.) circle (1.9804039991880444cm);
\draw (0.,6.)-- (0.,0.);
\draw (-3.68,0.)-- (4.,0.);
\draw (0.,0.)-- (0.,6.);
\draw (0.,0.)-- (4.,0.);
\draw (0.,0.)-- (-8.,0.);
\draw (0.,0.)-- (0.,-2.1);
\draw (0.,0.89)-- (0.8891187041486748,0.);
\draw [dotted] (0.,0.89)-- (0.89,0.89);
\draw [dotted] (0.89,0.89)-- (0.8891187041486748,0.);
\draw (0.89,0.89)-- (0.,0.);
\draw (0.9263774825946924,1.2567192648773278) node[anchor=north west] {$(a, a)$};
\draw [dotted] (0.,0.89)-- (-0.89,0.89);
\draw [dotted] (-0.89,0.89)-- (-0.89,-0.893);
\draw [dotted] (-0.89,-0.893)-- (0.89,-0.89);
\draw [dotted] (0.89,-0.89)-- (0.8891187041486748,0.);
\begin{scriptsize}
\draw [fill=xdxdff] (0.,6.) circle (1.5pt);
\draw[color=xdxdff] (-2.0791639787173755,3.39288922911565) node {$B_1$};
\draw [fill=xdxdff] (-3.68,0.) circle (1.5pt);
\draw[color=xdxdff] (-2.0791639787173755,3.39288922911565) node {$C_1$};
\draw [fill=xdxdff] (4.,0.) circle (1.5pt);
\draw[color=xdxdff] (4.105793243321508,0.21347346838884507) node {$D_1$};
\draw [fill=xdxdff] (-8.,0.) circle (1.5pt);
\draw[color=qqqqff] (-2.1040031643480535,3.355630450669633) node {$E$};
\draw [fill=xdxdff] (0.,-6.) circle (1.5pt);
\draw[color=xdxdff] (-2.1040031643480535,3.355630450669633) node {$F$};
\draw [fill=xdxdff] (0.,0.89) circle (1.5pt);
\draw[color=qqqqff] (-0.14170749952447215,1.120103743908598) node {$C$};
\draw [fill=xdxdff] (0.8891187041486748,0.) circle (1.5pt);
\draw[color=qqqqff] (1.0381538179327445,-0.1342751304406492) node {$D$};
\draw [fill=qqqqff] (0.89,0.89) circle (1.5pt);
\draw[color=qqqqff] (0.8642795185179968,1.0952645582779197) node {$P$};
\draw [fill=uuuuuu] (0.,0.) circle (1.5pt);
\draw[color=qqqqff] (-0.2185869548361193,-0.14668761766993682) node {$O$};
\draw [fill=qqqqff] (1.9804039991880442,0.) circle (1.5pt);
\draw[color=qqqqff] (2.1559171713132654,-0.1342751304406492) node {$A$};
\draw [fill=uuuuuu] (0.,1.9804039991880442) circle (1.5pt);
\draw[color=qqqqff] (0.1066843567823103,2.163349540397081) node {$B$};
\draw [fill=qqqqff] (-0.89,0.89) circle (1.5pt);
\draw[color=qqqqff] (-1.0234985894135498,1.107684151093259) node {$Q$};
\draw [fill=qqqqff] (-0.89,-0.893) circle (1.5pt);
\draw[color=qqqqff] (-1.0110789965982108,-1.0657445915910804) node {$R$};
\draw [fill=qqqqff] (0.89,-0.89) circle (1.5pt);
\draw[color=qqqqff] (1.000895039486727,-1.0409054059604022) node {$S$};
\draw[] (-0.0205869548361193 ,-2.5582947786169992) node { $(b)$ };
\end{scriptsize}
\end{tikzpicture}
\caption{$(a)$ $P$, $Q$, $R$, $S$ and $T$ form an optimal set of 5-means; $(b)$ $O$, $P$, $Q$, $R$ and $S$ do not form an optimal set of 5-means.} \label{FigB}
\end{figure}
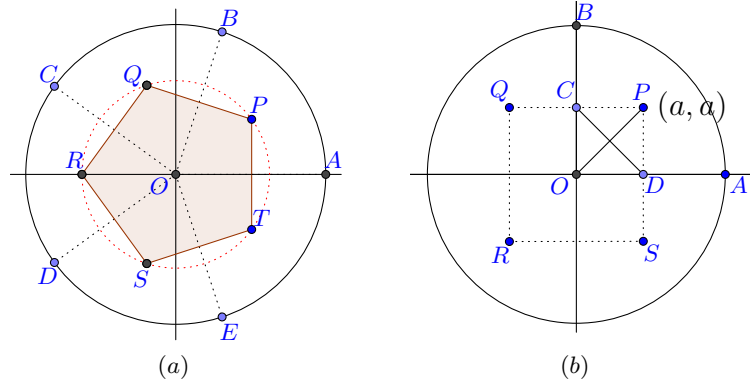

\tbf{Case 2.} In this case, we assume that one of the five optimal quantizers is the center $O$ of the disc, and the other four are the points $P$, $Q$, $R$  and $S$ equidistant from the center of the disc. Due to uniform distribution, the Voronoi regions of each of the points $P$, $Q$, $R$ and $S$ will subtend a central angle of $\frac \pi 2$ radians. As shown in Figure~\ref{FigB}$(b)$, let $ABCD$ be the Voronoi region of the point $P$, where the line $CD$ is the boundary of the Voronoi regions of the points $P$ and $O$. $A$ has the coordinates $(1, 0)$. $P$ being the centroid of its Voronoi region, the line segment $OP$ will be the perpendicular bisector of the line segment $CD$.  Let the equation of the line $CD$ be $x_1+x_2=a$ for some constant $a$. Then, the coordinates of $P$ are $(a, a)$. If $V(AOB)$ is the quantization error in the sector $AOB$ due to the points $O$ and $P$, then
\begin{align*}& V(AOB)\\
&=\int _0^a\int _0^{a-x_1}\frac{\left(x_1-0\right){}^2+\left(x_2-0\right){}^2}{\pi }dx_2dx_1+\int _0^a\int _{a-x_1}^{\sqrt{1-x_1^2}}\frac{\left(x_1-a\right){}^2+\left(x_2-a\right){}^2}{\pi }dx_2dx_1\\
&+\int _a^1\int _0^{\sqrt{1-x_1^2}}\frac{\left(x_1-a\right){}^2+\left(x_2-a\right){}^2}{\pi }dx_2dx_1\\
&=\frac{a^4}{6 \pi }+\frac{a \left(-2 a \left(a^2-3 \sqrt{1-a^2} a+6\right)+9 \sqrt{1-a^2}-8\right)+3 \left(4 a^2+1\right) \sin ^{-1}(a)}{12 \pi }\\
&+\frac{3 \left(4 a^2+1\right) \cos ^{-1}(a)-a \left(2 a \left(2 a^2+3 \sqrt{1-a^2} a-6\right)+9 \sqrt{1-a^2}+8\right)}{12 \pi }.
\end{align*}
It is easy to see that the above expression is minimum when $a=0.467885$, and the minimum value is $0.0307967$, i.e., $V(AOB)=0.0307967$. Thus, if $V_5(\te{Case 2})$ is the quantization error due to the five points in this case, we have
\[V_5(\te{Case 2})=4 V(AOB)=4 \times 0.0307967=0.123187.\]
Since $V_5(\te{Case 1})<V_5(\te{Case 2})$, the five points in Case 1 give the optimal set of five-means and hence, we can deduce the following theorem.

\begin{theorem}
For the uniform distribution on the unit disc, there are uncountably many optimal sets of five-means with quantization error $0.111049$. Any such five points form a regular pentagon inscribed in the circle $x_1^2+x_2^2=0.388951$.
\end{theorem}

\subsection{Optimal sets of 6-means}

\begin{figure}
\begin{tikzpicture}[line cap=round,line join=round,>=triangle 45,x=1.2cm,y=1.2cm]
\clip(-2.0,-2.0) rectangle (2.3,2.3);
\draw [shift={(0.,0.)},color=qqwuqq,fill=qqwuqq,fill opacity=0.1] (0,0) -- (0.:0.29807022756813895) arc (0.:71.99997284687713:0.29807022756813895) -- cycle;
\fill[color=zzttqq,fill=zzttqq,fill opacity=0.1] (1.11,-0.81) -- (1.1097,0.806241) -- (-0.42752924005170767,1.3054016190506708) -- (-1.3772891589038343,-0.0023411525305764526) -- (-0.427043829855083,-1.3097312529604481) -- cycle;
\draw(0.,0.) circle (2.3764847990256532cm);
\draw (0.,6.)-- (0.,0.);
\draw (-3.68,0.)-- (4.,0.);
\draw (0.,0.)-- (0.,6.);
\draw (0.,0.)-- (4.,0.);
\draw (0.,0.)-- (-8.,0.);
\draw (0.,0.)-- (0.,-2.1);
\draw (0.,0.)-- (1.9804039991880442,0.);
\draw (0.,0.)-- (0.6119793840779486,1.8834758383009789);
\draw (0.2619640029264046,0.8062409990491519)-- (0.8477320000000002,0.);
\draw [dotted] (0.,0.)-- (1.1097,0.806241);
\draw [dotted] (0.6119793840779486,1.8834758383009789)-- (0.6158876622112144,0.8096139235251195);
\draw [dotted] (0.8394403328873186,1.790761755936907)-- (0.8477320000000002,0.);
\draw [dotted] (0.2619640029264046,0.8062409990491519)-- (1.8017021729187577,0.8205965386593443);
\draw [color=zzttqq] (1.11,-0.81)-- (1.1097,0.806241);
\draw [color=zzttqq] (1.1097,0.806241)-- (-0.42752924005170767,1.3054016190506708);
\draw [color=zzttqq] (-0.42752924005170767,1.3054016190506708)-- (-1.3772891589038343,-0.0023411525305764526);
\draw [color=zzttqq] (-1.3772891589038343,-0.0023411525305764526)-- (-0.427043829855083,-1.3097312529604481);
\draw [color=zzttqq] (-0.427043829855083,-1.3097312529604481)-- (1.11,-0.81);
\draw [dotted] (0.,0.) circle (1.6489414786462253cm);
\begin{scriptsize}
\draw [fill=xdxdff] (0.,6.) circle (1.5pt);
\draw[color=xdxdff] (-4.047371780286109,2.6204316125067324) node {$B_1$};
\draw [fill=xdxdff] (-3.68,0.) circle (1.5pt);
\draw[color=xdxdff] (-3.5903307646816294,0.16632007219572978) node {$C_1$};
\draw [fill=xdxdff] (4.,0.) circle (1.5pt);
\draw[color=xdxdff] (4.089945432324085,0.16632007219572978) node {$D_1$};
\draw [fill=xdxdff] (-8.,0.) circle (1.5pt);
\draw[color=xdxdff] (-4.067243128790651,2.5906245897499187) node {$E$};
\draw [fill=xdxdff] (0.,-6.) circle (1.5pt);
\draw[color=xdxdff] (-4.067243128790651,2.5906245897499187) node {$F$};
\draw [fill=uuuuuu] (0.,0.) circle (1.5pt);
\draw[color=qqqqff] (-0.12290910213440332,-0.14181448112013694) node {$O$};
\draw [fill=uuuuuu] (1.9804039991880442,0.) circle (1.5pt);
\draw[color=qqqqff] (2.1226819303743674,-0.11187880686786568) node {$A$};
\draw [fill=xdxdff] (0.6119793840779486,1.8834758383009789) circle (1.5pt);
\draw[color=qqqqff] (0.6621378152904868,2.0540981801272706) node {$B$};
\draw[color=qqwuqq] (0.3441962392178052,0.24580546621389993) node {$72\textrm{\degre}$};
\draw [fill=xdxdff] (0.8477320000000002,0.) circle (1.5pt);
\draw[color=qqqqff] (0.9005939973449979,-0.19117122364284963) node {$D$};
\draw [fill=xdxdff] (0.2619640029264046,0.8062409990491519) circle (1.5pt);
\draw[color=qqqqff] (0.13554707992010787,0.9413026638728885) node {$C$};
\draw [fill=qqqqff] (1.1097,0.806241) circle (1.5pt);
\draw[color=qqqqff] (1.1887285506608656,0.9810453608819736) node {$P$};
\draw [fill=qqqqff] (0.8394403328873186,1.790761755936907) circle (1.5pt);
\draw [fill=qqqqff] (0.6158876622112144,0.8096139235251195) circle (1.5pt);
\draw [fill=qqqqff] (1.8017021729187577,0.8205965386593443) circle (1.5pt);
\draw [fill=qqqqff] (1.11,-0.81) circle (1.5pt);
\draw[color=qqqqff] (1.2284712476699509,-0.8669900500404819) node {$T$};
\draw [fill=uuuuuu] (-0.42752924005170767,1.3054016190506708) circle (1.5pt);
\draw[color=qqqqff] (-0.3612366326934571,1.4480220507387231) node {$Q$};
\draw [fill=uuuuuu] (-1.3772891589038343,-0.0023411525305764526) circle (1.5pt);
\draw[color=qqqqff] (-1.5137748459569278,0.16632007219572978) node {$R$};
\draw [fill=uuuuuu] (-0.427043829855083,-1.3097312529604481) circle (1.5pt);
\draw[color=qqqqff] (-0.3811079811979997,-1.423387808167673) node {$S$};
\end{scriptsize}
\end{tikzpicture}
\caption{$O$, $P$, $Q$, $R$, $S$ and $T$ form an optimal set of 6-means.} \label{FigC}
\end{figure}
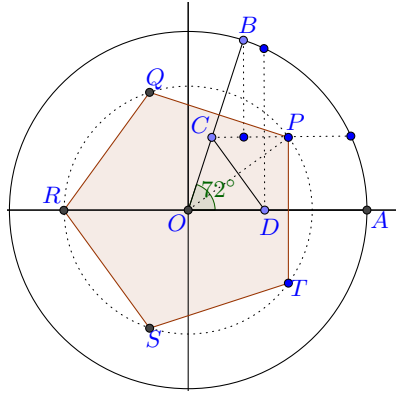

In this section, we calculate the optimal sets of 6-means. By  Conjecture~\ref{conj1}, one of the following two cases can arise:

\tbf{Case 1.} In this case, we assume that the six optimal quantizers are equidistant from the center of the disc. So, the Voronoi regions of the points partition the disc into six sectors each making a central angle of $\frac{\pi}{3}$ radians. Let $P$ be one of the six optimal quantizers whose Voronoi region is the sector
$AOB$ where $O$ is the origin, and $A$, $B$ are the points $(1, 0)$, $(\frac 12, \frac{\sqrt{3}}{2})$. Let $V(P)$ be the quantization error due to the point $P$ in the sector $AOB$. Then, we have
\begin{align*}
\tilde p&=\frac{\int _{\frac{1}{2}}^1\int _0^{\sqrt{1-x_1^2}}\frac{(x_1,x_2)}{\pi }dx_2dx_1+\int _0^{\frac{1}{2}}\int _0^{x_1 \tan \left(\frac{\pi }{3}\right)}\frac{(x_1,x_2)}{\pi }dx_2dx_1}{\int _{\frac{1}{2}}^1\int _0^{\sqrt{1-x_1^2}}\frac{1}{\pi }dx_2dx_1+\int _0^{\frac{1}{2}}\int _0^{x_1 \tan \left(\frac{\pi }{3}\right)}\frac{1}{\pi }dx_2dx_1}=(0.551329,0.31831),
\end{align*}
and
\begin{align*}
V(P)&=\int _{\frac{1}{2}}^1\int _0^{\sqrt{1-x_1^2}}\frac{\left(x_1-0.551329\right){}^2+\left(x_2-0.31831\right){}^2}{\pi }dx_2dx_1\\
&+\int _0^{\frac{1}{2}}\int _0^{x_1 \tan \left(\frac{\pi }{3}\right)}\frac{\left(x_1-0.551329\right){}^2+\left(x_2-0.31831\right){}^2}{\pi }dx_2dx_1\\
&=0.0157859.
\end{align*}
If $V_6(\te{Case 1})$ is the quantization error due to the six points in this case, we have
$V_6(\te{Case 1})=6\times 0.0157859=0.0947153.$

\tbf{Case 2.} In this case, we assume that one of the six optimal quantizers is the center $O$ of the disc, and the other five are equidistant from the center of the disc. Due to the uniform distribution, the Voronoi regions of each of these five points will subtend a central angle of $\frac {2\pi} 5$ radians. As shown in Figure~\ref{FigC}, let $ABCD$ be the Voronoi region of such a point $P$, where $A$ has the coordinates $(1, 0)$ and the line $CD$ is the boundary of the Voronoi regions of the points $P$ and $O$. Let $OC=OD=a$ for some constant $a$. Then, the coordinates of $C$ and $D$ are, respectively, $(a\cos\frac{2\pi}{5}, a\sin\frac{2\pi}{5})$ and $(a, 0)$, and the equation of the line $CD$ is $x_2=-(x_1-a)\cot{\frac \pi 5}$. Thus, we have the coordinates of $P$ as $(2a \cos^2\frac\pi 5, a\sin\frac{2\pi}{5})$. Let $\tilde m$ represent the position vector of the centroid of the sector $AOB$, and let $V(AOB)$ be the quantization error of the sector $AOB$ due to the points $O$ and $P$. Then,
\begin{align*} \tilde m
&=\frac{\int _{\cos \left(\frac{2 \pi }{5}\right)}^1\int _0^{\sqrt{1-x_1^2}}\frac{(x_1,x_2)}{\pi }dx_2dx_1+\int _0^{\cos \left(\frac{2 \pi }{5}\right)}\int _0^{x_1 \tan \left(\frac{2 \pi }{5}\right)}\frac{(x_1,x_2)}{\pi }dx_2dx_1}{\int _{\cos \left(\frac{2 \pi }{5}\right)}^1\int _0^{\sqrt{1-x_1^2}}\frac{1}{\pi }dx_2dx_1+\int _0^{\cos \left(\frac{2 \pi }{5}\right)}\int _0^{x_1 \tan \left(\frac{2 \pi }{5}\right)}\frac{1}{\pi }dx_2dx_1}=(0.504551,0.366578), \\
\tilde p &=\frac{(\te{Area of the sector } AOB) \,\tilde m-(\te{Area of the } \tri  OCD)(\te{Centroid of the } \tri OCD)}{\te{Area of the sector } AOB-\te{Area of the } \tri OCD}\\
&=\frac{\frac{1}{5} \pi  (0.504551, \, 0.366578)-\frac{a^2 \sin \left(\frac{2 \pi }{5}\right) \left(2 a \cos ^2\left(\frac{\pi }{5}\right), \,  a \sin \left(\frac{2 \pi }{5}\right)\right)}{2\cdot 3}}{\frac{\pi }{5}-\frac{1}{2} a^2 \sin \left(\frac{2 \pi }{5}\right)}\\
&=\left (\frac{0.666667\, -0.436339 a^3}{1.32131\, -a^2},\frac{0.484362\, -0.317019 a^3}{1.32131\, - a^2}\right).
\end{align*}
Since the line segments $OP$ and $CD$ bisect each other, we have
\[\frac{0.484362\, -0.317019 a^3}{1.32131\, -a^2}=a \sin \left(\frac{2 \pi }{5}\right) \te{which yields } a=0.423866,\]
and hence, using Figure~\ref{FigC}, we have
\begin{align*}& V(AOB)=\int _0^{a \cos \left(\frac{2 \pi }{5}\right)}\int _0^{x_1 \tan \left(\frac{2 \pi }{5}\right)}\frac{\left(x_1-0\right){}^2+\left(x_2-0\right){}^2}{\pi }dx_2dx_1\\
&+\int _{a \cos \left(\frac{2 \pi }{5}\right)}^a\int _0^{\cot \left(\frac{\pi }{5}\right) \left(-\left(x_1-a\right)\right)}\frac{\left(x_1-0\right){}^2+\left(x_2-0\right){}^2}{\pi }dx_2dx_1\\
&+\int _{a \cos \left(\frac{2 \pi }{5}\right)}^a\int _{\cot \left(\frac{\pi }{5}\right) \left(-\left(x_1-a\right)\right)}^{a \sin \left(\frac{2 \pi }{5}\right)}\frac{\left(x_2-a \sin \left(\frac{2 \pi }{5}\right)\right){}^2+\left(x_1-2 a \cos ^2\left(\frac{\pi }{5}\right)\right){}^2}{\pi }dx_2dx_1\\
&+\int _a^1\int _0^{\sqrt{1-x_1^2}}\frac{\left(x_2-a \sin \left(\frac{2 \pi }{5}\right)\right){}^2+\left(x_1-2 a \cos ^2\left(\frac{\pi }{5}\right)\right){}^2}{\pi }dx_2dx_1\\
&+\int _{a \cos \left(\frac{2 \pi }{5}\right)}^{\cos \left(\frac{2 \pi }{5}\right)}\int _{a \sin \left(\frac{2 \pi }{5}\right)}^{x_1 \tan \left(\frac{2 \pi }{5}\right)}\frac{\left(x_2-a \sin \left(\frac{2 \pi }{5}\right)\right){}^2+\left(x_1-2 a \cos ^2\left(\frac{\pi }{5}\right)\right){}^2}{\pi }dx_2dx_1\\
&+\int _{\cos \left(\frac{2 \pi }{5}\right)}^a\int _{a \sin \left(\frac{2 \pi }{5}\right)}^{\sqrt{1-x_1^2}}\frac{\left(x_2-a \sin \left(\frac{2 \pi }{5}\right)\right){}^2+\left(x_1-2 a \cos ^2\left(\frac{\pi }{5}\right)\right){}^2}{\pi }dx_2dx_1\\
&=0.018719.
\end{align*}
Thus, if $V_6(\te{Case 2})$ is the quantization error due to the six points in this case, we have
\[V_6(\te{Case 2})=5\times 0.018719=0.093595.\]
Since $V_6(\te{Case 2})<V_6(\te{Case 1})$, the six points in Case 2 give the optimal set of six-means. Moreover, the coordinates of $P$ are $(0.554847, 0.40312)$, and hence, we can deduce the following theorem.

\begin{theorem}
For the uniform distribution on the unit disc, there are uncountably many optimal sets of six-means with quantization error $0.093595$. One of such six points is the center of the disc, and the other five form a regular pentagon inscribed in the circle $x_1^2+x_2^2=0.470361$.
\end{theorem}

\section{Optimal sets and quantization error for uniform distribution on a square}

In this section, we determine the optimal sets of $n$-means and the $n$th quantization errors for all $1\leq n\leq 5$ for a uniform distribution defined on a region with uniform density whose boundary forms a square. Let $X:=(X_1, X_2)$ be a bivariate continuous random variable with the uniform distribution taking values on the region whose boundary forms a square with vertices $ O (0, 0), \, A (1, 0), \, B (1, 1)$ and $C (0, 1)$. In the sequel by the term `square' we mean the square region and denote it by $J$. As mentioned in the previous section, let $f(x_1, x_2)$ denote the pdf of the random variable $X$, $f_1(x_1)$ and $f_2(x_2)$ represent the marginal pdfs of the random variables $X_1$ and $X_2$, respectively. We also keep the other notations same as before. Thus, we have
\[f(x_1, x_2)=\left\{\begin{array}{ccc}
1 & \te{ for }  (x_1, x_2) \in J, \\
 0  & \te{ otherwise},
\end{array}\right.
\]
\[f_1(x_1)=\left\{\begin{array}{ccc}
1 & \te{ for }   0<x_1<1, \\
0   & \te{ otherwise},
\end{array}\right.
\te{ and }
f_2(x_2)=\left\{\begin{array}{ccc}
1 & \te{ for }   0<x_2<1, \\
0   & \te{ otherwise}.
\end{array}\right.
\]
Let us now prove the following lemma.
\begin{lemma}
Let $X:=(X_1, X_2)$ be a bivariate continuous random variable with uniform distribution taking values on the square $J$. Then,
\[E(X)=(E(X_1), E(X_2))=(\frac 1 2, \frac12) \te{ and } V(X)=V(X_1)+V(X_2)=\frac 1 {6}.\]
\end{lemma}
\begin{proof}
$E(X_1)  =\int_0^1 x_1 \, dx_1=\frac 1 2$, $E(X_2) =\int_0^1 x_2 \, dx_2=\frac 1 2$, $E(X_1^2) =\int_0^1 x_1^2 \, dx_1=\frac 1 3$, and $E(X_2^2) =\int_0^1 x_2^2 \, dx_2=\frac 1 3$.  Then, \begin{align*}
E(X)& =\iint(x_1 i+x_2 j) f(x_1, x_2) dx_1dx_2 =i \int x_1 f_1(x_1) dx_1+j \int x_2 f_2(x_2) dx_2\\
& =(E(X_1), E(X_2))=(\frac 1 2, \frac{1}{2}),\\
V(X_1)& =E(X_1^2)-[E(X_1)]^2=\frac{1}{12} \te{ and } V(X_2)=E(X_2^2)-[E(X_2)]^2=\frac{1}{12}.
\end{align*}
Thus, we have
\[V(X)=E\|X-E(X)\|^2=\iint \Big((x_1-E(X_1))^2 +(x_2-E(X_2))^2\Big)f(x_1, x_2)\, dx_1 dx_2,\]
implying
\[V(X)=\int (x_1-E(X_1))^2f_1(x_1)\,dx_1+\int (x_2-E(X_2))^2f_2(x_2)\,dx_2=V(X_1)+V(X_2)=\frac{1}{6}.\]
Thus, the lemma is yielded.
\end{proof}
\begin{note} \label{note1} For any $(a, b) \in \D R^2$, since \begin{align*} & E\|X-(a, b)\|^2=V(X_1)+V(X_2)+(a-\frac 1 2)^2+(b-\frac 1 2)^2=\frac 1 6+\|(a, b)-(\frac 12, \frac 1 2)\|^2,
\end{align*}
we see that the optimal set of one-mean consists of the expected vector $(\frac 12, \frac 1 2)$ of the random variable $X$, which is the centroid of the square $J$ and the corresponding quantization error is  $\frac 1 {6}$ which is the expected squared distance of the random variable $X$.
\end{note}

\begin{figure}
\begin{tikzpicture}[line cap=round,line join=round,>=triangle 45,x=3.0cm,y=3.0cm]
\clip(-0.1,-0.2) rectangle (1.2,1.2);
\draw (0.,1.)-- (1.,1.);
\draw (1.,0.)-- (1.,1.);
\draw (1.,0.)-- (0.,0.);
\draw (0.,0.)-- (0.,1.);
\draw (0.71,1.)-- (0.29215656532836076,0.);
\begin{scriptsize}
\draw [fill=xdxdff] (0.,1.) circle (1.5pt);
\draw[color=qqqqff] (-4.2665609340125296E-4,1.0663184606346283) node {$C$};
\draw [fill=qqqqff] (1.,1.) circle (1.5pt);
\draw[color=qqqqff] (1.0454027311163012,1.0663184606346283) node {$B$};
\draw [fill=xdxdff] (1.,0.) circle (1.5pt);
\draw[color=qqqqff] (1.0178530809640356, -0.06288855043308171) node {$A$};
\draw [fill=xdxdff] (0.,0.) circle (1.5pt);
\draw[color=qqqqff] (-0.037777705636604916,-0.06288855043308171) node {$O$};
\draw [fill=xdxdff] (0.71,1.) circle (1.5pt);
\draw[color=qqqqff] (0.7528195096945393,1.0663184606346283) node {$F$};
\draw [fill=xdxdff] (0.29215656532836076,0.) circle (1.5pt);
\draw[color=qqqqff] (0.27970621548062624,-0.06488855043308171) node {$E$};
\draw [fill=qqqqff] (0.2859313904044935,0.5883044667252469) circle (1.5pt);
\draw[color=qqqqff] (0.32950761487156444,0.6754569156593887) node {$P$};
\draw [fill=qqqqff] (0.71,0.41) circle (1.5pt);
\draw[color=qqqqff] (0.7528195096945393,0.49492684286723787) node {$Q$};
\end{scriptsize}
\end{tikzpicture}
\caption{Two points $P$ and $Q$.} \label{FigD}
\end{figure}
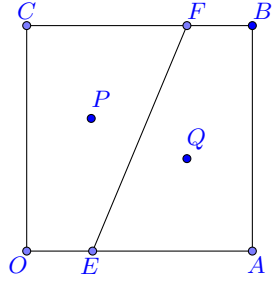

\subsection{Optimal sets of 2-means}
In this subsection, we investigate the optimal sets of two-means and the corresponding quantization error.
Let us divide $J$ by a straight line $\ell$ into two regions as shown in Figure~\ref{FigD}. Let $\ell$ intersect the side $OA$ at the point $E$ and the side $BC$ at the point $F$. It might be that the line $\ell$ is a diagonal of the square. Thus, the square $J$ is divided into two regions: the quadrilateral $OEFC$ and the quadrilateral $EABF$.  Let the position vectors of $A, \, B, \, C, \, E, \, F$ be denoted, respectively, by $\tilde a$, $\tilde b$, $\tilde c$, $\tilde e$ and $\tilde f$. We know $\tilde a=(1, 0)$, $\tilde b=(1, 1)$ and $\tilde c=(0, 1)$. Let the lengths of $OE$ and $CF$ be, respectively, $\ga$ and $\gb$. Then, we have  $ \tilde e=\ga \tilde a, \  \tilde f=\gb \tilde b+(1-\gb)\tilde c$. Now, the area of the triangles $OEC$, $ECF$, $EAF$ and $ABF$ are, respectively, $\frac 1 2 \ga$, $\frac 12 \gb$, $\frac 1 2 (1-\ga)$ and $\frac 1 2(1-\gb)$. By Remark~\ref{rem1}, we know that the points in an optimal set are the centroids of their corresponding Voronoi regions. Let $\tilde p$ and $\tilde q$ be the position vectors of the centroids $P$ and $Q$ of the quadrilaterals $OEFC$ and $EABF$, respectively.
Then, taking moments about the origin, we have
\[\tilde p=\frac{\frac 1 2 \alpha \frac 1 3  (\tilde c+\tilde e)+\frac 1 2\beta \frac 13 (\tilde c+\tilde e+\tilde f)}{\frac 1 2 \ga+\frac 1 2\gb}=\Big(\frac{\alpha ^2+\alpha  \beta +\beta ^2}{3 (\alpha +\beta )},\frac{\alpha +2 \beta }{3 (\alpha +\beta )}\Big ),\]
and
\[\tilde q=\frac{\frac 1 2 (1-\ga)\frac{1}{3} (\tilde a+\tilde e+\tilde f)+\frac 1 2 (1-\gb)\frac{1}{3} (\tilde a+\tilde b+\tilde f)}{\frac{1}{2} (2-\alpha -\beta)}=\Big(\frac{\alpha ^2+\alpha  \beta +\beta ^2-3}{3 (\alpha +\beta -2)},\frac{\alpha +2 \beta -3}{3 (\alpha +\beta -2)}\Big).\]
If $\tilde p$ and $\tilde q$ are the optimal quantizers, we must have,
\begin{equation*}\rho(\tilde p, \tilde e)=\rho(\tilde q,\tilde e) \te{ and } \rho(\tilde p,\tilde f)=\rho(\tilde q,\tilde f).\end{equation*}
Solving the above two equations, we have three possible sets of solutions for $\ga$ and $\gb$: $\{\alpha=1,\beta =0\},\, \{\alpha = 0,\beta =1\}, \, \{\alpha = \frac{1}{2}, \beta = \frac{1}{2}\}$. If  $\{\alpha=1,\beta =0\}$, then $\tilde p= (\frac{1}{3},\frac{1}{3})$, $\tilde q= (\frac{2}{3},\frac{2}{3})$, and the corresponding distortion error is
\[\int _0^1\int _0^{1-x_1}\Big((x_1-\frac{1}{3}){}^2+(x_2-\frac{1}{3}){}^2\Big)dx_2dx_1+\int _0^1\int _{1-x_1}^1\Big((x_1-\frac{2}{3}){}^2+(x_2-\frac{2}{3}){}^2\Big)dx_2dx_1=\frac 1 9.\]
Similarly, if  $\{\alpha=0,\beta =1\}$, then $\tilde p= (\frac{1}{3},\frac{2}{3})$,  $\tilde q= (\frac{2}{3},\frac{1}{3})$, and the corresponding distortion error is $\frac 1 9$. If $\{\alpha = \frac{1}{2}, \beta = \frac{1}{2}\}$, then $\tilde p=(\frac 1 4, \frac 1 2)$ and $\tilde q= (\frac 3 4, \frac 12)$, and the corresponding distortion error is
\begin{align*} \int_0^{\frac{1}{2}}\int_0^1\Big((x_1-\frac{1}{4})^2+(x_2-\frac{1}{2})^2\Big)dx_2dx_1
+\int_{\frac{1}{2}}^1\int_0^1\Big((x_1-\frac{3}{4})^2+(x_2-\frac{1}{2})^2\Big)dx_2dx_1=\frac{5}{48}.\end{align*}
Due to symmetry, we can say that the points $(\frac 12, \frac 14)$ and $(\frac 12, \frac 34)$ will also give the same distortion error $\frac 5{48}$.
Since $\frac{5}{48}<\frac 1 9$, we can deduce the following theorem.
\begin{theorem}
For the uniform distribution on the square $J$ with vertices $(0, 0)$, $(1, 0)$, $(1, 1)$ and $(0, 1)$, there are two optimal sets of two-means $\set{(\frac 1 4, \frac 1 2),  ( \frac 3 4, \frac 12)}$ and $\set{(\frac 1 2, \frac 1 4),  ( \frac 12, \frac 34)}$ with quantization error $\frac{5}{48}$.
\end{theorem}

We now give the following note.
\begin{note} \label{note101} With respect to the uniform distribution the unit square has four lines of maximum symmetry: the two diagonals, and the two lines which bisect the opposite sides of the square. By the maximum symmetry it is meant that with respect to any of the above lines the unit square is geometrically symmetric, also symmetric with respect to the uniform distribution. By the `symmetric with respect to the uniform distribution' it is meant that if two regions of similar geometrical shape are equidistant from any of the above four lines and are in opposite sides of the line, then they have the same probability. This motivates us to give the following conjecture.

\begin{conj} \label{conj2} $(i)$ The points in an optimal sets of three-means are symmetrically distributed over the square with respect to one of the two diagonals of the square, or with respect to one of the two lines which are perpendicular bisectors of the opposite sides of the square; $(ii)$  the two points in an optimal set of four-means lie on a line of symmetry of the square, and the other two are in opposite sides and equidistant from the line of symmetry; $(iii)$ one point in an optimal set of five-means is at the center $(\frac 12, \frac 12)$ of the
square; and the other four lie on the lines $x_1=\frac 12$ and $x_2=\frac 12$ and equidistant from the center of the square, or lie on the two diagonals of the square and equidistant from the center of the square.
\end{conj}
Under the above conjecture, in the following subsections, we determine the optimal sets of $n$-means for $n=3, 4,$ and $5$. It is extremely difficult and the answer is not known yet how the points in an optimal sets of $n$-means for higher values of $n$ are located.
\end{note}

\subsection{Optimal sets of 3-means}
In this subsection, we determine the optimal sets of three-means. By Conjecture~\ref{conj2}, the following two cases can arise:

\begin{figure} \begin{tikzpicture}[line cap=round,line join=round,>=triangle 45,x=3.0cm,y=3.0cm]
\clip(-0.2,-0.3) rectangle (1.2,1.3);
\draw (0.,1.)-- (0.,0.);
\draw (0.,0.)-- (1.,0.);
\draw (1.,1.)-- (1.,0.);
\draw (1.,1.)-- (0.,1.);
\draw (0.,1.)-- (0.,0.);
\draw (1.,1.)-- (0.,0.);
\draw (0.3481831396431663,1.)-- (0.508777,0.508777);
\draw (1.,0.378352)-- (0.508777,0.508777);
\begin{scriptsize}
\draw [fill=xdxdff] (0.,1.) circle (1.5pt);
\draw[color=qqqqff] (-0.04022805548433947,1.0704567143853105) node {$C$};
\draw [fill=uuuuuu] (0.,0.) circle (1.5pt);
\draw[color=qqqqff] (-0.029102180865003086,-0.05911372535694898) node {$O$};
\draw [fill=xdxdff] (1.,0.) circle (1.5pt);
\draw[color=qqqqff] (1.039177556192434,-0.040438200585347174) node {$A$};
\draw [fill=qqqqff] (1.,1.) circle (1.5pt);
\draw[color=qqqqff] (1.0454027311163012,1.0704567143853105) node {$B$};
\draw [fill=xdxdff] (0.7313255000000001,0.7313255000000001) circle (1.5pt);
\draw[color=qqqqff] (0.8088460840093448,0.7190331401264596) node {$P$};
\draw [fill=qqqqff] (0.203336,0.592112) circle (1.5pt);
\draw[color=qqqqff] (0.24858034086128983,0.651682090583256) node {$Q$};
\draw [fill=qqqqff] (0.592112,0.203336) circle (1.5pt);
\draw[color=qqqqff] (0.634541186141061,0.25949607037961805) node {$R$};
\draw [fill=xdxdff] (0.508777,0.508777) circle (1.5pt);
\draw[color=qqqqff] (0.63096496167399,0.5409534171820434) node {$M$};
\draw [fill=xdxdff] (0.3481831396431663,1.) circle (1.5pt);
\draw[color=qqqqff] (0.3917593641102372,1.0704567143853105) node {$D$};
\draw [fill=qqqqff] (1.,0.378352) circle (1.5pt);
\draw[color=qqqqff] (1.0516279060401685,0.451350618400167) node {$E$};
\draw[] (0.50096496167399,-0.160438200585347174) node { $(a)$ };
\end{scriptsize}
\end{tikzpicture}\qquad
\begin{tikzpicture}[line cap=round,line join=round,>=triangle 45,x=3.0cm,y=3.0cm]
\clip(-0.2,-0.3) rectangle (1.2,1.3);
\draw (0.,1.)-- (1.,1.);
\draw (1.,0.)-- (1.,1.);
\draw (0.,1.)-- (0.,0.);
\draw (1.,0.)-- (0.,0.);
\draw (0.5,0.)-- (0.5,1.);
\draw (6.552678678847919E-4,0.7648095641098724)-- (0.5,0.5);
\draw (0.9977664630287412,0.7547883460680548)-- (0.5,0.5);
\begin{scriptsize}
\draw [fill=xdxdff] (0.,1.) circle (1.5pt);
\draw[color=qqqqff] (-0.01938716821575037,1.0704567143853105) node {$C$};
\draw [fill=qqqqff] (1.,1.) circle (1.5pt);
\draw[color=qqqqff] (1.0328407261751031,1.0704567143853105) node {$B$};
\draw [fill=xdxdff] (1.,0.) circle (1.5pt);
\draw[color=qqqqff] (1.027851335196012,-0.03383483413099496) node {$A$};
\draw [fill=uuuuuu] (0.,0.) circle (1.5pt);
\draw[color=qqqqff] (-0.04441899527847685,-0.056887879235539066) node {$O$};
\draw [fill=xdxdff] (0.5,0.) circle (1.5pt);
\draw[color=qqqqff] (0.5217586060424028,-0.056887879235539066) node {$G$};
\draw [fill=xdxdff] (0.5,1.) circle (1.5pt);
\draw[color=qqqqff] (0.5367904331051293,1.0704567143853105) node {$H$};
\draw [fill=xdxdff] (0.5,0.803764) circle (1.5pt);
\draw[color=qqqqff] (0.5367904331051293,0.8750429625698665) node {$P$};
\draw [fill=qqqqff] (0.231591,0.317285) circle (1.5pt);
\draw[color=qqqqff] (0.17602658359969386,0.26875927103989905) node {$Q$};
\draw [fill=qqqqff] (0.768409,0.317285) circle (1.5pt);
\draw[color=qqqqff] (0.8123739292551146,0.2787804890817167) node {$R$};
\draw [fill=qqqqff] (6.552678678847919E-4,0.7648095641098724) circle (1.5pt);
\draw[color=qqqqff] (-0.05446143136211215,0.7948732182353253) node {$D$};
\draw [fill=qqqqff] (0.9977664630287412,0.7547883460680548) circle (1.5pt);
\draw[color=qqqqff] (1.0358407261751031,0.8249368723607782) node {$E$};
\draw [fill=qqqqff] (0.5,0.5) circle (1.5pt);
\draw[color=qqqqff] (0.5667692150633116,0.4641730228553431) node {$M$};
\draw[] (0.50096496167399,-0.160438200585347174) node { $(b)$ };
\end{scriptsize}
\end{tikzpicture}
\caption{$(a)$ Three points $P$, $Q$ and $R$ in Case 1 (not an optimal configuration); $(b)$ three points $P$, $Q$ and $R$ in Case 2 (optimal configuration).} \label{FigE}
\end{figure}
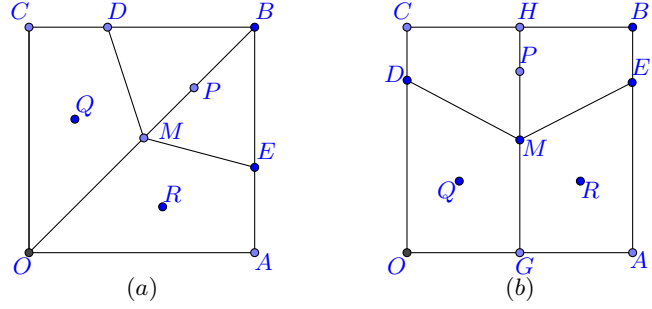

\tbf{Case 1.}  As shown in Figure~\ref{FigE}$(a)$, in this case we assume that one of the three points in an optimal set of three-means lies on the diagonal $OB$, and the other two are on either side of the diagonal $OB$.
Let the boundaries of the three Voronoi regions meet at the point $M$ on the diagonal, and cut the side $BC$ and $AB$ at the points $D$ and $E$, respectively. Let $BD=\ga$ and $BM=\gb \sqrt 2$. Then, due to symmetry, we have $BE=\ga$. Then, the position vectors of $D$, $E$ and $M$ are given by,
$\tilde d=(1-\alpha ) \tilde b+\alpha  \tilde c, \, \tilde e=\alpha \tilde a  +(1-\alpha ) \tilde b \te{ and } \tilde m=(1-\beta ) \tilde b.$
$\te{Area of the triangle } BMD=\te{Area of the triangle } BME=\frac 12 \ga \gb.$
Let $P$, $Q$, $R$ be the centroids of the quadrilaterals $MEBD$, $OMDC$, and $OAEM$. Then, we have
\begin{align*} \tilde p&=\frac{\frac 12 \ga\gb \frac 1 3   (\tilde b+\tilde d+\tilde m)+\frac 1 2 \alpha  \beta \frac 1 3  (\tilde b+\tilde e+\tilde m)}{\alpha  \beta }=\Big(\frac{1}{6} (-\alpha -2 \beta +6), \frac{1}{6} (-\alpha -2 \beta +6)\Big),\\
\tilde q &=\frac{\frac 1 2 \cdot\frac 1 3  (\tilde b+\tilde c)-\frac 1 2 \alpha  \beta  \frac 1 3 (\tilde b+\tilde d+\tilde m)}{\frac{1-\alpha  \beta }{2}}=\Big(-\frac{\alpha ^2 \beta +\alpha  (\beta -3) \beta +1}{3 \alpha  \beta -3},-\frac{\alpha  (\beta -3) \beta +2}{3 \alpha  \beta -3}\Big),\\
\tilde r &=\frac{\frac 1 2 \cdot\frac 1 3  (\tilde b+\tilde a)-\frac 1 2 \alpha  \beta  \frac 1 3 (\tilde b+\tilde e+\tilde m)}{\frac{1-\alpha  \beta }{2}}=\Big (-\frac{\alpha  (\beta -3) \beta +2}{3 \alpha  \beta -3},-\frac{\alpha ^2 \beta +\alpha  (\beta -3) \beta +1}{3 \alpha  \beta -3}\Big).
\end{align*}
If $\tilde p$, $\tilde q$ and $\tilde r$ form an optimal set of three-means, we must have \[\rho(\tilde p, \tilde d)=\rho(\tilde q, \tilde d) \te{ and } \rho(\tilde  p,  \tilde m)=\rho(\tilde q,  \tilde m).\]
Solving the above two equations, we have $\{\alpha = 0.621648,\beta =0.491223\}$, and so
\[\tilde p= (0.732651, 0.732651), \, \tilde q=(0.203336, 0.592112) \te{ and } \tilde r=(0.592112,  0.203336).\]
Moreover, we have $\tilde m=(0.508777,0.508777)$, $\tilde d=(0.378352,1)$, and so the equation of the line $MD$ is $x_2=0.508777-3.76633 (x_1-0.508777)$. Now, if $V_3(\te{Case 1})$ is the distortion error in this case, we have
\begin{align*}
&V_3(\te{Case 1})=2\times (\te{Distortion error in the triangle } OBC)\\
&=2 \times \Big(\int _0^{0.378352}\int _{x_1}^1\left(\left(x_2-0.592112\right){}^2+\left(x_1-0.203336\right){}^2\right)dx_2dx_1\\
&+\int _{0.378352}^{0.508777}\int _{x_1}^{0.508777\, -3.76633 \left(x_1-0.508777\right)}\left(\left(x_2-0.592112\right){}^2+\left(x_1-0.203336\right){}^2\right)dx_2dx_1\\
&+\int _{0.378352}^{0.508777}\int _{0.508777\, -3.76633 \left(x_1-0.508777\right)}^1\left(\left(x_1-0.732651\right){}^2+\left(x_2-0.732651\right){}^2\right)dx_2dx_1\\
& +\int _{0.508777}^1\int _{x_1}^1\left(\left(x_1-0.732651\right){}^2+\left(x_2-0.732651\right){}^2\right)dx_2dx_1\Big)\\
&=2 \times 0.0332909=0.0665818.
\end{align*}
 \tbf{Case 2.} As shown in Figure~\ref{FigE}$(b)$, in this case we assume that one of the three points in an optimal set of three-means lies on the line $GH$ which bisects the opposite sides $OA$ and $BC$ at the points $G$ and $H$, and the other two points are equidistant from $GH$.
Let the boundaries of the three Voronoi regions meet at the point $M$ on the line $GH$, and cut the sides $OC$ and $AB$ at the points $D$ and $E$, respectively. Let $OD=\ga$ and $GM=\gb$. Then, due to symmetry, we have $AE=\ga$. Then,
Area of the triangle $OMD=$ Area of the triangle  AME $=\frac 1 4 \ga$,  Area of the triangle $OGM=$ Area of the triangle $AGM=\frac 1 4 \gb$,  Area of the triangle $ MDH=$ Area of the triangle $MEH=\frac 1 4 (1-\gb)$,  and Area of the triangle $ CDH=$ Area of the triangle $ BEH=\frac 1 4 (1-\ga)$. Moreover, the position vectors of the points $D, E, G, H$, and $M$ are, respectively, given by $\tilde d=\ga \tilde c$, $\tilde e=\ga \tilde b+(1-\ga) \tilde a$, $\tilde g=(\frac 12, 0)$, $\tilde h=(\frac 12, 1)$, and $\tilde m=\gb \tilde h+(1-\gb) \tilde g$.
Thus,
\begin{align*}
&\tilde p=\frac{\frac{(1-\alpha ) (b+e+h)}{3\cdot 4}+\frac{(1-\alpha ) (c+d+h)}{3\cdot 4}+\frac{(1-\beta ) (d+h+m)}{3\cdot 4}+\frac{(1-\beta ) (e+h+m)}{3\cdot 4}}{\frac{1}{2} (-\alpha -\beta +2)}=\Big(\frac{1}{2},\frac{\alpha ^2+\alpha  \beta +\beta ^2-3}{3 (\alpha +\beta -2)}\Big),\\
\tilde q&=\frac{\frac{\alpha  (d+m)}{3\ 4}+\frac{\beta  (g+m)}{3\ 4}}{\frac{\alpha +\beta }{4}}=\Big(\frac{\alpha +2 \beta }{6 (\alpha +\beta )},\frac{\alpha ^2+\alpha  \beta +\beta ^2}{3 (\alpha +\beta )}\Big), \\
\tilde r&=\frac{\frac{\alpha  (a+e+m)}{3\ 4}+\frac{\beta  (a+g+m)}{3\ 4}}{\frac{\alpha +\beta }{4}}=\Big(\frac{5 \alpha +4 \beta }{6 (\alpha +\beta )},\frac{\alpha ^2+\alpha  \beta +\beta ^2}{3 (\alpha +\beta )}\Big).
\end{align*}
If $\tilde p$, $\tilde q$ and $\tilde r$ form an optimal set of 3-means, we must have
\[\rho(\tilde p,  \tilde d)=\rho(\tilde  q,  \tilde  d) \te{ and }
\rho(\tilde  p,  \tilde  m)=\rho(\tilde   q,  \tilde m).\]
Solving the above two equations, we have
$\{\alpha = 0.762348,\beta= 0.486479\}$, and so
\[\tilde p=(0.5, 0.803764) , \, \tilde q=(0.231591, 0.317285), \te{ and } \tilde r=(0.768409, 0.317285).\]
Moreover, we have $\tilde m=(\frac{1}{2}, 0.486479)$, $\tilde d=(0, 0.762348)$, and so the equation of the line $MD$ is $x_2=0.486479\, -0.551737 (x_1-0.5)$. Now, if $V_3(\te{Case 2})$ is the distortion error in this case, we have
\begin{align*}
&V_3(\te{Case 2})=2\times (\te{Distortion error in the rectangle } OGHC)\\
&=2 \times \Big(\int _0^{\frac{1}{2}}\int _{0.486479\, -0.551737 (x_1-0.5)}^1\Big((x_2-0.803764){}^2+(x_1-0.5){}^2\Big)dx_2dx_1\\
&+\int _0^{\frac{1}{2}}\int _0^{0.486479\, -0.551737 (x_1-0.5)}\Big((x_2-0.317285){}^2+(x_1-0.231591){}^2\Big)dx_2dx_1\Big)\\
&=2 \times 0.0330899=0.0661797.
\end{align*}
Since $V_3(\te{Case 2})< V_3(\te{Case 1})$, we see that the points in Case 2 give the error minimum. Hence the points $P$, $Q$ and $R$ in Case 2 form an optimal set of three-means. Notice that due to rotational symmetry there are three more optimal sets of three-means. Thus, we deduce the following theorem.
\begin{theorem}
For the uniform distribution on the square there are four optimal sets of three-means with quantization error $0.0661797$. One of the four optimal sets of three-means is the set
$\set{(0.5, 0.803764), \, (0.231591, 0.317285), \, (0.768409, 0.317285)}$.
\end{theorem}

\subsection{Optimal set of 4-means}

In this subsection, we calculate the optimal set of four-means. By Conjecture~\ref{conj2}, the following two cases can arise:

 \tbf{Case 1.}   As shown in Figure~\ref{FigF}$(a)$, in this case we assume that two of the four points, say $P$ and $Q$ lie on the diagonal $OB$, and the other two, say $R$ and $S$ are equidistant from the diagonal $OB$, in fact due to symmetry $R$ and $S$ lie on the other diagonal $AC$. Thus, the boundaries of the Voronoi regions of the four optimal quantizers are the lines $x_1=\frac 12$ and $x_2=\frac 12$, i.e., the Voronoi regions of the optimal quantizers in this case divide the square $J$ into four congruent squares. Due to symmetry, we have
\[\tilde p=(\frac 1 4, \frac 14), \, \tilde q=(\frac 34, \frac 34), \, \tilde r=(\frac 34, \frac 14),  \te{ and } \tilde s=(\frac 14, \frac 34), \]
and the corresponding distortion error is given by
\begin{align*} &4\times (\te{Distortion error in the Voronoi region of } P)\\
&=4 \times \Big( \int _0^{\frac{1}{2}}\int _0^{\frac{1}{2}}\Big((x_1-\frac{1}{4}){}^2+(x_2-\frac{1}{4}){}^2\Big)dx_2dx_1\Big)=4 \times \frac 1{96}=0.0416667.
\end{align*}

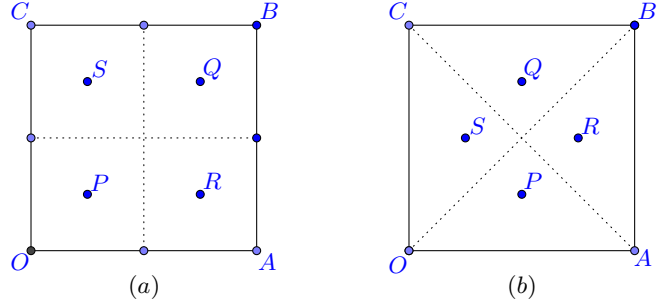
\begin{figure}
\begin{tikzpicture}[line cap=round,line join=round,>=triangle 45,x=3.0cm,y=3.0cm]
\clip(-0.2,-0.3) rectangle (1.2,1.3);
\draw (0.,1.)-- (0.,0.);
\draw (0.,0.)-- (1.,0.);
\draw (1.,0.)-- (1.,1.);
\draw (1.,1.)-- (0.,1.);
\draw [dotted] (0.,0.5)-- (1.,0.5);
\draw [dotted] (0.5,1.)-- (0.5,0.);
\begin{scriptsize}
\draw [fill=xdxdff] (0.,1.) circle (1.5pt);
\draw[color=qqqqff] (-0.04862534625887355,1.0626837060992063) node {$C$};
\draw [fill=uuuuuu] (0.,0.) circle (1.5pt);
\draw[color=qqqqff] (-0.04862534625887355,-0.051044215130043346) node {$O$};
\draw [fill=xdxdff] (1.,0.) circle (1.5pt);
\draw[color=qqqqff] (1.0467224433855023,-0.048720241447018184) node {$A$};
\draw [fill=qqqqff] (1.,1.) circle (1.5pt);
\draw[color=qqqqff] (1.0567224433855023,1.0626597324161811) node {$B$};
\draw [fill=xdxdff] (0.5,1.) circle (1.5pt);
\draw [fill=xdxdff] (0.5,0.) circle (1.5pt);
\draw [fill=xdxdff] (0.,0.5) circle (1.5pt);
\draw [fill=qqqqff] (1.,0.5) circle (1.5pt);
\draw [fill=qqqqff] (0.75,0.75) circle (1.5pt);
\draw[color=qqqqff] (0.8034135749253328,0.8070508639560116) node {$Q$};
\draw [fill=qqqqff] (0.25,0.75) circle (1.5pt);
\draw[color=qqqqff] (0.30447186432196877,0.8070508639560116) node {$S$};
\draw [fill=qqqqff] (0.25,0.25) circle (1.5pt);
\draw[color=qqqqff] (0.30447186432196877,0.3010915335264776) node {$P$};
\draw [fill=qqqqff] (0.75,0.25) circle (1.5pt);
\draw[color=qqqqff] (0.8034135749253328,0.30810915335264776) node {$R$};
\draw[] (0.50096496167399,-0.160438200585347174) node { $(a)$ };
\end{scriptsize}
\end{tikzpicture}\qquad
\begin{tikzpicture}[line cap=round,line join=round,>=triangle 45,x=3.0cm,y=3.0cm]
\clip(-0.2,-0.3) rectangle (1.2,1.3);
\draw (0.,1.)-- (0.,0.);
\draw (0.,0.)-- (1.,0.);
\draw (1.,0.)-- (1.,1.);
\draw (1.,1.)-- (0.,1.);
\draw [dotted] (0.,0.)-- (1.,1.);
\draw [dotted] (0.,1.)-- (1.,0.);
\begin{scriptsize}
\draw [fill=xdxdff] (0.,1.) circle (1.5pt);
\draw[color=qqqqff] (-0.04862534625887355,1.0626837060992063) node {$C$};
\draw [fill=uuuuuu] (0.,0.) circle (1.5pt);
\draw[color=qqqqff] (-0.04862534625887355,-0.062644215130043346) node {$O$};
\draw [fill=xdxdff] (1.,0.) circle (1.5pt);
\draw[color=qqqqff] (1.0367224433855023,-0.032620241447018184) node {$A$};
\draw [fill=qqqqff] (1.,1.) circle (1.5pt);
\draw[color=qqqqff] (1.0567224433855023,1.0626597324161811) node {$B$};
\draw [fill=xdxdff] (0.,0.) circle (1.5pt);
\draw [fill=qqqqff] (1.,1.) circle (1.5pt);
\draw [fill=qqqqff] (0.5,0.75) circle (1.5pt);
\draw[color=qqqqff] (0.5501047064651634,0.8070508639560116) node {$Q$};
\draw [fill=qqqqff] (0.25,0.5) circle (1.5pt);
\draw[color=qqqqff] (0.30447186432196877,0.5537419954958423) node {$S$};
\draw [fill=qqqqff] (0.5,0.25) circle (1.5pt);
\draw[color=qqqqff] (0.5501047064651634,0.30810915335264776) node {$P$};
\draw [fill=qqqqff] (0.75,0.5) circle (1.5pt);
\draw[color=qqqqff] (0.8034135749253328,0.5537419954958423) node {$R$};
\draw[] (0.50096496167399,-0.160438200585347174) node { $(b)$ };
\end{scriptsize}
\end{tikzpicture}
\caption{$(a)$ Four points $P$, $Q$, $R$ and $S$ in Case 1 (optimal configuration); $(b)$ four points $P$, $Q$, $R$ and $S$ in Case 2 (not an optimal configuration).} \label{FigF}
\end{figure}

\tbf{Case 2.}   As shown in Figure~\ref{FigF}$(b)$, in this case we assume that two of the four optimal quantizers, say $P$ and $Q$, lie on the line of
symmetry $x_1=\frac 12$ of the square and the other two, say $R$ and $S$, lie on the line of symmetry $x_2=\frac 12$ of the square. Thus, the boundaries of the Voronoi regions in this case are the two diagonals of the square, in other words, the Voronoi regions in this case partition the square $J$ into four congruent triangles. Due to symmetry, we have
\[\tilde p=(\frac 1 2, \frac 14), \, \tilde q=(\frac 12, \frac 34), \, \tilde r=(\frac 34, \frac 12),  \te{ and } \tilde s=(\frac 14, \frac 12),\]
and the corresponding distortion error is given by
\begin{align*} &4\times (\te{Distortion error in the Voronoi region of } P)\\
&=4 \times \Big( \int _0^{\frac{1}{2}}\int _{x_2}^{1-x_2}\Big((x_1-\frac{1}{2}){}^2+(x_2-\frac{1}{4}){}^2\Big)dx_1dx_2\Big)
=4 \times \frac 1{64}=0.0625.
\end{align*}

Notice that the error in Case 1 is less than the error in Case 2. Thus, we deduce the following theorem.
\begin{theorem}
For the uniform distribution on the unit square with vertices $(0, 0)$, $(1, 0)$, $(1, 1)$ and $(0, 1)$, the set $\set{(\frac 1 4, \frac 14), \, (\frac 34, \frac 34), \, (\frac 34, \frac 14),  (\frac 14, \frac 34)}$ forms a unique optimal set of four-means with quantization error $0.0416667$.
\end{theorem}

\subsection{Optimal set of 5-means}

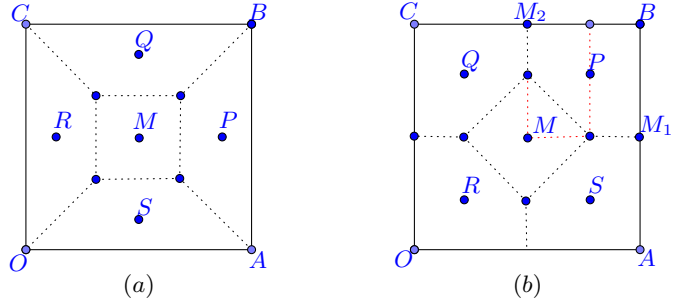
\begin{figure}
\begin{tikzpicture}[line cap=round,line join=round,>=triangle 45,x=3.0cm,y=3.0cm]
\clip(-0.2,-0.3) rectangle (1.2,1.3);
\draw (0.,1.)-- (0.,0.);
\draw (0.,0.)-- (1.,0.);
\draw (1.,0.)-- (1.,1.);
\draw (1.,1.)-- (0.,1.);
\draw [dotted] (0.3104558126320994,0.6830340757908577)-- (0.6856550950737798,0.6830340757908577);
\draw [dotted] (0.6856550950737798,0.6830340757908577)-- (0.6816636133456768,0.31581775680538365);
\draw [dotted] (0.6816636133456768,0.31581775680538365)-- (0.3104558126320994,0.31182627507728067);
\draw [dotted] (0.3104558126320994,0.6830340757908577)-- (0.3104558126320994,0.31182627507728067);
\draw [dotted] (1.,1.)-- (0.6856550950737798,0.6830340757908577);
\draw [dotted] (1.,0.)-- (0.6816636133456768,0.31581775680538365);
\draw [dotted] (0.,1.)-- (0.3104558126320994,0.6830340757908577);
\draw [dotted] (0.3104558126320994,0.31182627507728067)-- (0.,0.);
\begin{scriptsize}
\draw [fill=xdxdff] (0.,1.) circle (1.5pt);
\draw[color=qqqqff] (-0.024828652528551074,1.0502503947763318) node {$C$};
\draw [fill=uuuuuu] (0.,0.) circle (1.5pt);
\draw[color=qqqqff] (-0.03281161598475704,-0.057449671811472537) node {$O$};
\draw [fill=xdxdff] (1.,0.) circle (1.5pt);
\draw[color=qqqqff] (1.0289225236906363,-0.043458190083369558) node {$A$};
\draw [fill=qqqqff] (1.,1.) circle (1.5pt);
\draw[color=qqqqff] (1.0289225236906363,1.0564418765044348) node {$B$};
\draw [fill=xdxdff] (0.,0.) circle (1.5pt);
\draw [fill=qqqqff] (1.,1.) circle (1.5pt);
\draw [fill=qqqqff] (0.5,0.866025) circle (1.5pt);
\draw[color=qqqqff] (0.5259958259496604,0.9225229794770364) node {$Q$};
\draw [fill=qqqqff] (0.5,0.133975) circle (1.5pt);
\draw[color=qqqqff] (0.5259958259496604,0.1880903415060883) node {$S$};
\draw [fill=qqqqff] (0.87,0.5) circle (1.5pt);
\draw[color=qqqqff] (0.8972036266632378,0.5753066604915624) node {$P$};
\draw [fill=qqqqff] (0.133975,0.5) circle (1.5pt);
\draw[color=qqqqff] (0.1627709886922891,0.5753066604915624) node {$R$};
\draw [fill=qqqqff] (0.6856550950737798,0.6830340757908577) circle (1.5pt);
\draw [fill=qqqqff] (0.6816636133456768,0.31581775680538365) circle (1.5pt);
\draw [fill=qqqqff] (0.3104558126320994,0.31182627507728067) circle (1.5pt);
\draw [fill=qqqqff] (0.3104558126320994,0.6830340757908577) circle (1.5pt);
\draw [fill=qqqqff] (0.5020469355810426,0.4954344345700177) circle (1.5pt);
\draw[color=qqqqff] (0.5299873076777634,0.5713151787634594) node {$M$};
\draw[] (0.50096496167399,-0.160438200585347174) node { $(a)$ };
\end{scriptsize}
\end{tikzpicture}
\qquad
\begin{tikzpicture}[line cap=round,line join=round,>=triangle 45,x=3.0cm,y=3.0cm]
\clip(-0.2,-0.3) rectangle (1.2,1.3);
\draw (0.,1.)-- (0.,0.);
\draw (0.,0.)-- (1.,0.);
\draw (1.,0.)-- (1.,1.);
\draw (1.,1.)-- (0.,1.);
\draw [dotted] (0.7774591748201484,0.5034173980262237)-- (0.5020469355810426,0.7748381555372262);
\draw [dotted] (0.21865173288573084,0.49942591629812066)-- (0.5020469355810426,0.7748381555372262);
\draw [dotted] (0.21865173288573084,0.49942591629812066)-- (0.4940639721248366,0.21603071360280918);
\draw [dotted] (0.4940639721248366,0.21603071360280918)-- (0.7774591748201484,0.5034173980262237);
\draw [dotted] (0.5,1.)-- (0.5020469355810426,0.7748381555372262);
\draw [dotted] (0.4940639721248366,0.21603071360280918)-- (0.4980554538529396,-0.0035007814428546654);
\draw [dotted] (0.21865173288573084,0.49942591629812066)-- (-8.797621599331807E-4,0.5034173980262237);
\draw [dotted] (0.7774591748201484,0.5034173980262237)-- (0.9969906698658124,0.49942591629812066);
\draw [dotted,color=ffqqqq] (0.7774591748201484,0.5034173980262237)-- (0.7774591748201488,1.);
\draw [dotted,color=ffqqqq] (0.5020469355810426,0.7748381555372262)-- (0.5020469355810426,0.4954344345700177);
\draw [dotted,color=ffqqqq] (0.7774591748201484,0.5034173980262237)-- (0.5020469355810426,0.4954344345700177);
\begin{scriptsize}
\draw [fill=xdxdff] (0.,1.) circle (1.5pt);
\draw[color=qqqqff] (-0.02482865252855103,1.0582333582325374) node {$C$};
\draw [fill=uuuuuu] (0.,0.) circle (1.5pt);
\draw[color=qqqqff] (-0.046811615984757,-0.042449671811472995) node {$O$};
\draw [fill=xdxdff] (1.,0.) circle (1.5pt);
\draw[color=qqqqff] (1.0289225236906367,-0.037458190083370016) node {$A$};
\draw [fill=qqqqff] (1.,1.) circle (1.5pt);
\draw[color=qqqqff] (1.0289225236906367,1.0564418765044343) node {$B$};
\draw [fill=xdxdff] (0.,0.) circle (1.5pt);
\draw [fill=qqqqff] (1.,1.) circle (1.5pt);
\draw [fill=qqqqff] (0.221063,0.778937) circle (1.5pt);
\draw[color=qqqqff] (0.25058358671055486,0.8347103814587704) node {$Q$};
\draw [fill=qqqqff] (0.778937,0.221063) circle (1.5pt);
\draw[color=qqqqff] (0.8053995469168697,0.2859029395243534) node {$S$};
\draw [fill=qqqqff] (0.778937,0.778937) circle (1.5pt);
\draw[color=qqqqff] (0.8053995469168697,0.8347103814587704) node {$P$};
\draw [fill=qqqqff] (0.221063,0.221063) circle (1.5pt);
\draw[color=qqqqff] (0.25058358671055486,0.2859029395243534) node {$R$};
\draw [fill=qqqqff] (0.7774591748201484,0.5034173980262237) circle (1.5pt);
\draw [fill=qqqqff] (0.4940639721248366,0.21603071360280918) circle (1.5pt);
\draw [fill=qqqqff] (0.21865173288573084,0.49942591629812066) circle (1.5pt);
\draw [fill=qqqqff] (0.5020469355810426,0.4954344345700177) circle (1.5pt);
\draw[color=qqqqff] (0.5799873076777638,0.5513151787634589) node {$M$};
\draw [fill=qqqqff] (0.9969906698658124,0.49942591629812066) circle (1.5pt);
\draw[color=qqqqff] (1.0729140054187398,0.5513811056758709) node {$M_1$};
\draw [fill=qqqqff] (0.5,1.) circle (1.5pt);
\draw [fill=qqqqff] (-8.797621599331807E-4,0.5034173980262237) circle (1.5pt);
\draw [fill=qqqqff] (0.5020469355810426,0.7748381555372262) circle (1.5pt);
\draw [fill=xdxdff] (0.7774591748201488,1.) circle (1.5pt);
\draw[color=qqqqff] (0.5140213807653519,1.054273730329258) node {$M_2$};
\draw[] (0.50096496167399,-0.160438200585347174) node { $(b)$ };
\end{scriptsize}
\end{tikzpicture}
\caption{$(a)$ Five points $M$, $P$, $Q$, $R$ and $S$ in Case 1 (not an optimal configuration); $(b)$ five points $M$, $P$, $Q$, $R$ and $S$ in Case 2 (optimal configuration).} \label{FigG}
\end{figure}

In this subsection, we calculate the optimal set of five-means. By Conjecture~\ref{conj2}, the following two cases can arise:

 \tbf{Case 1.}  As shown in Figure~\ref{FigG}$(a)$, in this case we assume that one point in an optimal set of five-means is at the center $M (\frac 12, \frac 12)$ of the
square, and the other four lie on the lines $x_1=\frac 12$ and $x_2=\frac 12$, and are equidistant from the center of the square.
Then, for some positive scalar $\ga$, we can consider the optimal set of 5-means
as $\set{M (\frac 12, \frac 12), \, P (\frac 1 2 +\ga, \frac 12), \, Q (\frac 12, \frac 12 +\ga), \, R (\frac 1 2 -\ga, \frac 12), \, S(\frac 12, \frac 12 -\ga)}$. Thus, the distortion error in the triangle $MAB$ due to the points $M$ and $P$ is given by
 \begin{align*}
&\int _{\frac{1}{2}}^{\frac{\alpha +1}{2}}\int _{1-x_1}^{x_1}((x_1-\frac{1}{2}){}^2+(x_2-\frac{1}{2}){}^2)dx_2dx_1+\int _{\frac{\alpha +1}{2}}^1\int _{1-x_1}^{x_1}((x_1-(\alpha +\frac{1}{2})){}^2+(x_2-\frac{1}{2}){}^2)dx_2dx_1\\
&=\frac{1}{24} (-2 \alpha ^4+6 \alpha ^2-4 \alpha +1),
 \end{align*}
which is minimum when $\ga=0.366025$, and the minimum value is $0.0126603$.
Thus, if $V_5(\te{Case }1)$ is the distortion error in this case, we have $V_5(\te{Case }1)=4\times 0.0126603=0.0506413.$

 \tbf{Case 2.}   As shown in Figure~\ref{FigG}$(b)$, in this case we assume that one point lies at the center of the square, and the other four lie on the two diagonals of the square and are equidistant from $(\frac 12, \frac 12)$. Then, for some positive scalar $\ga$, we can consider the optimal set of 5-means as
$\set{M(\frac 12, \frac 12), \, P(\ga, \ga), \, Q(1-\ga, \ga), \, R(1-\ga, 1-\ga), \, S(\ga, 1-\ga)}$. The boundary of the Voronoi regions of the points $M$ and $P$ is given by the equation $x_2=-x_1+\frac{1+2\ga}{2}$. Thus, the distortion error in the quadrilateral $MM_1BM_2$ due to the points $M$ and $P$ is given by
 \begin{align*}
&\int _{\frac{1}{2}}^{\alpha }\int _{\frac{1}{2}}^{\frac{1}{2} (2 \alpha +1)-x_1}((x_1-\frac{1}{2}){}^2+(x_2-\frac{1}{2}){}^2)dx_2dx_1+\int _{\frac{1}{2}}^{\alpha }\int _{\frac{1}{2} (2 \alpha +1)-x_1}^1((x_1-\alpha ){}^2+(x_2-\alpha ){}^2)dx_2dx_1\\
&+\int _{\alpha }^1\int _{\frac{1}{2}}^1((x_1-\alpha ){}^2+(x_2-\alpha ){}^2)dx_2dx_1\\
&=-\frac{\alpha ^4}{2}+\frac{5 \alpha ^3}{3}-2 \alpha ^2+\frac{1}{24} \left(-16 \alpha ^3+42 \alpha ^2-37 \alpha +11\right)+\frac{25 \alpha }{24}+\frac{1}{96} (1-2 \alpha )^4-\frac{19}{96},
 \end{align*}
which is minimum when $\ga=0.778937$, and the minimum value is $0.00881742$. Thus, if $V_5(\te{Case }2)$ is the distortion error in this case, we have $V_5(\te{Case }2)=4\times 0.00881742=0.0352697.$

 Since $V_5(\te{Case }2)< V_5(\te{Case }1)$, the points in Case 2 give the optimal set five-means. Thus, we deduce the following theorem.
\begin{theorem}
For the uniform distribution on the unit square with vertices $(0, 0)$, $(1, 0)$, $(1, 1)$ and $(0, 1)$, the set $\set{(0.5, 0.5) , \, (0.778937, 0.778937), \, (0.221063, 0.778937), \, (0.221063, 0.221063), \, \\ (0.778937, 0.221063)}$ forms a unique optimal set of five-means with quantization error $0.0352697$.
\end{theorem}

\section{Optimal quantization for a nonuniform distribution}
In this section, due to calculation simplicity we determine the optimal sets of $n$-means and the $n$th quantization errors for all $1\leq n\leq 4$ for a nonuniform distribution $P$ on the real line with probability density function (pdf) $f(x)$ given by
\[f(x)=\left\{\begin{array}{ccc}
\frac 25 & \te{ for } 0\leq x\leq \frac 12, \\
\frac 85 & \te{ for } \frac 12< x\leq 1, \\
 0  & \te{ otherwise}.
\end{array}\right.
\]
 Notice that for the given pdf $f(x)$, we have  $P([0, \frac 12])=\frac 1 5$ and $P([\frac 12, 1])=\frac 45$. First, we give the following lemma.

\begin{lemma}
Let $X$ be a random variable with distribution $P$. Let $E(X)$ and $V(X)$ represent the expected value and the variance of the random variable $X$. Then, $E(X)=\frac{13}{20}$ and $V(X)=\frac{73}{1200}$.
\end{lemma}

\begin{proof} $E(X)=\mathop{\int}\limits_0^{\frac 12} \frac 25 xdx+\mathop{\int}\limits_{\frac 12}^1 \frac 85 xdx=\frac{13}{20}$ and $E(X^2)=\mathop{\int}\limits_0^{\frac 12} \frac 25 x^2dx+\mathop{\int}\limits_{\frac 12}^1 \frac 85 x^2dx=\frac{29}{60}$ yielding $V(X)=E(X^2)-[E(X)]^2=\frac{29}{60}-(\frac{13}{20})^2=\frac{73}{1200}$, which is the lemma.
\end{proof}

\begin{remark} For any $a\in \D R$, $E(X-a)^2=V(X)+(a-E(X))^2$ which yields the fact that the optimal set of one-mean is the expected value $\frac{13}{20}$ and the corresponding quantization error is the variance $V(X)$ of the random variable $X$.
\end{remark}

\begin{prop}\label{prop43}
Let $P$ be the probability distribution with pdf $f(x)$. Then, $\set{\frac{11}{32}, \frac{25}{32}}$ forms a unique optimal set of two-means with quantization error $\frac{317}{15360}=0.020638$.
\end{prop}

\begin{proof} Let $\ga:=\set{a, b}$ be an optimal set of two-means such that $a<b$. Since the optimal quantizers are the centroids of their own Voronoi regions, we have $0<a<b<1$. By the properties of centroids, we have
\[a P(M(a|\ga))+b P(M(b|\ga))=E(X)=\frac{13}{20}=0.65,\]
which implies that the two optimal quantizers $a$ and $b$ lie in the opposite sides of the point $\frac{13}{20}$.
Thus, the following three cases can arise:

Case~1. $\frac 12 \leq a<b<1$.

Then, the distortion error is given by
\begin{align*}
&\int_0^{\frac{1}{2}} \frac{2}{5} (x-a)^2 \, dx+\int_{\frac{1}{2}}^{\frac{a+b}{2}} \frac{8}{5} (x-a)^2 \, dx+\int_{\frac{a+b}{2}}^1 \frac{8}{5} (x-b)^2 \, dx\\
& =\frac{2 a^3}{5}+\frac{2 a^2 b}{5}-\frac{3 a^2}{5}-\frac{2 a b^2}{5}+\frac{3 a}{10}-\frac{2 b^3}{5}+\frac{8 b^2}{5}-\frac{8 b}{5}+\frac{29}{60},
\end{align*}
which is minimum when $a=\frac 12$ and $b=\frac 56$, and the minimum value is $\frac{13}{540}=0.0240741$.

Case~2. $0<a\leq \frac 12\leq \frac{a+b}{2}<b<1$.

Then, the distortion error is given by
\begin{align*}
&\int_0^{\frac{1}{2}} \frac{2}{5} (x-a)^2  dx+\int_{\frac{1}{2}}^{\frac{a+b}{2}} \frac{8}{5} (x-a)^2  dx+\int_{\frac{a+b}{2}}^1 \frac{8}{5} (x-b)^2  dx\\
& =\frac{2 a^3}{5}+\frac{2 a^2 b}{5}-\frac{3 a^2}{5}-\frac{2 a b^2}{5}+\frac{3 a}{10}-\frac{2 b^3}{5}+\frac{8 b^2}{5}-\frac{8 b}{5}+\frac{29}{60},
\end{align*}
which is minimum when $a=\frac{11}{32}$ and $b=\frac{25}{32}$, and the minimum value is $\frac{317}{15360}=0.020638$.

Case~3. $0<a<\frac{a+b}{2}\leq \frac 12\leq b<1$.

Then, the distortion error is given by
\begin{align*}
&\int_0^{\frac{a+b}{2}} \frac{2}{5} (x-a)^2  dx+\int_{\frac{a+b}{2}}^{\frac{1}{2}} \frac{2}{5} (x-b)^2  dx+\int_{\frac{1}{2}}^1 \frac{8}{5} (x-b)^2  dx\\
&=\frac{a^3}{10}+\frac{a^2 b}{10}-\frac{a b^2}{10}-\frac{b^3}{10}+b^2-\frac{13 b}{10}+\frac{29}{60},
\end{align*}
which is minimum when $a=\frac{1}{4}$ and $b=\frac{3}{4}$, and the minimum value is $\frac{1}{48}=0.0208333$.

Comparing the distortion errors in all the above possible cases, we see that $\set{\frac{11}{32}, \frac{25}{32}}$ forms a unique optimal set of two-means with quantization error $\frac{317}{15360}=0.020638$. Thus, the proposition is yielded.
\end{proof}

\begin{prop} \label{prop44}
Let $P$ be the probability distribution with pdf $f(x)$. Then, \\
$\set{0.200339, 0.601018,  0.867006}$ forms a unique optimal set of three-means with quantization error $0.00739237$.
\end{prop}

\begin{proof} Let $\ga:=\set{a, b, c}$ be an optimal set of three-means.  Since the optimal quantizers are the centroids of their own Voronoi regions, we have $0<a<b<c<1$. By the properties of centroids, we have
\[a P(M(a|\ga))+b P(M(b|\ga))+cP(M(c|\ga))=E(X)=\frac{13}{20}=0.65,\]
which implies that all the optimal quantizers can not lie in one side of the point $\frac{13}{20}$. Thus, the following cases can arise:

Case~1. $0<a<\frac{a+b}{2}<b<\frac {b+c}{2}\leq \frac 12 \leq c<1$.

Then, the distortion error is given by
\begin{align*}
&\int_0^{\frac{a+b}{2}} \frac{2}{5} (x-a)^2  dx+\int_{\frac{a+b}{2}}^{\frac{b+c}{2}} \frac{2}{5}  (x-b)^2  dx+\int_{\frac{b+c}{2}}^{\frac{1}{2}} \frac{2}{5} (x-c)^2  dx+\int_{\frac{1}{2}}^1 \frac{8}{5} (x-c)^2  dx\\
& =\frac{a^3}{10}+\frac{a^2 b}{10}-\frac{a b^2}{10}+\frac{b^2 c}{10}-\frac{b c^2}{10}-\frac{c^3}{10}+c^2-\frac{13 c}{10}+\frac{29}{60},
\end{align*}
which is minimum when $a=0.0874081, \, b=0.262224$, and $c=0.737776$, and the minimum value is $0.0188458$.

Case~2.  $0<a<\frac{a+b}{2}<b\leq \frac{1}{2}\leq \frac{b+c}{2}<c<1$.

Then, the distortion error is given by
\begin{align*}
&\int_0^{\frac{a+b}{2}} \frac{2}{5} (x-a)^2 \, dx+\int_{\frac{a+b}{2}}^{\frac{1}{2}} \frac{2}{5} (x-b)^2 \, dx+\int_{\frac{1}{2}}^{\frac{b+c}{2}} \frac{8}{5} (x-b)^2 \, dx+\int_{\frac{b+c}{2}}^1 \frac{8}{5} (x-c)^2 \, dx\\
& =\frac{a^3}{10}+\frac{a^2 b}{10}-\frac{a b^2}{10}+\frac{3 b^3}{10}+\frac{2 b^2 c}{5}-\frac{3 b^2}{5}-\frac{2 b c^2}{5}+\frac{3 b}{10}-\frac{2 c^3}{5}+\frac{8 c^2}{5}-\frac{8 c}{5}+\frac{29}{60},
\end{align*}
which is minimum when $a=\frac 1{6}, \, b=\frac 1{2}$, and $c=\frac 56$, and the minimum value is $\frac{1}{108}=0.00925926$.

Case~3.  $0<a<\frac{a+b}{2}\leq \frac{1}{2}\leq b< \frac{b+c}{2}<c<1$.

Then, the distortion error is given by
\begin{align*}
&\int_0^{\frac{a+b}{2}} \frac{2}{5} (x-a)^2 \, dx+\int_{\frac{a+b}{2}}^{\frac{1}{2}} \frac{2}{5} (x-b)^2 \, dx+\int_{\frac{1}{2}}^{\frac{b+c}{2}} \frac{8}{5} (x-b)^2 \, dx+\int_{\frac{b+c}{2}}^1 \frac{8}{5} (x-c)^2 \, dx\\
& =\frac{a^3}{10}+\frac{a^2 b}{10}-\frac{a b^2}{10}+\frac{3 b^3}{10}+\frac{2 b^2 c}{5}-\frac{3 b^2}{5}-\frac{2 b c^2}{5}+\frac{3 b}{10}-\frac{2 c^3}{5}+\frac{8 c^2}{5}-\frac{8 c}{5}+\frac{29}{60},
\end{align*}
which is minimum when $a=0.200339$, $b=0.601018$, and $c=0.867006$, and the minimum value is $0.00739237$.

Case~4.  $0<a\leq \frac{1}{2}\leq \frac{a+b}{2}<b< \frac{b+c}{2}<c<1$.

Then, the distortion error is given by
\begin{align*}
&\int_0^{\frac{1}{2}} \frac{2}{5} (x-a)^2 \, dx+\int_{\frac{1}{2}}^{\frac{a+b}{2}} \frac{8}{5} (x-a)^2 \, dx+\int_{\frac{a+b}{2}}^{\frac{b+c}{2}} \frac{8}{5} (x-b)^2 \, dx+\int_{\frac{b+c}{2}}^1 \frac{8}{5} (x-c)^2 \, dx\\
& =\frac{2 a^3}{5}+\frac{2 a^2 b}{5}-\frac{3 a^2}{5}-\frac{2 a b^2}{5}+\frac{3 a}{10}+\frac{2 b^2 c}{5}-\frac{2 b c^2}{5}-\frac{2 c^3}{5}+\frac{8 c^2}{5}-\frac{8 c}{5}+\frac{29}{60},
\end{align*}
which is minimum when $a=0.325207$, $b=0.674793$, and $c=0.891598$, and the minimum value is $0.0101842$.

Case~5.  $0<\frac 12 \leq a <\frac{a+b}{2}<b< \frac{b+c}{2}<c<1$.

Then, the distortion error is given by
\begin{align*}
&\int_0^{\frac{1}{2}} \frac{2}{5} (x-a)^2 \, dx+\int_{\frac{1}{2}}^{\frac{a+b}{2}} \frac{8}{5} (x-a)^2 \, dx+\int_{\frac{a+b}{2}}^{\frac{b+c}{2}} \frac{8}{5} (x-b)^2 \, dx+\int_{\frac{b+c}{2}}^1 \frac{8}{5} (x-c)^2 \, dx\\
& =\frac{2 a^3}{5}+\frac{2 a^2 b}{5}-\frac{3 a^2}{5}-\frac{2 a b^2}{5}+\frac{3 a}{10}+\frac{2 b^2 c}{5}-\frac{2 b c^2}{5}-\frac{2 c^3}{5}+\frac{8 c^2}{5}-\frac{8 c}{5}+\frac{29}{60},
\end{align*}
which is minimum when $a=\frac 12$, $b=\frac 7{10}$, and $c=\frac 9{10}$, and the minimum value is $\frac{29}{1500}=0.0193333$.

Comparing the distortion errors in all the above possible cases, we see that \\
$\set{0.200339, 0.601018,  0.867006}$ forms a unique optimal set of three-means with quantization error $0.00739237$. Hence, the proof of the proposition is complete.
\end{proof}

\begin{prop} \label{prop45}
Let $P$ be the probability distribution with pdf $f(x)$. Then, $\set{\frac{59}{332}, \frac{177}{332},   \frac{239}{332},  \frac{301}{332}}$ forms a unique optimal set of four-means with quantization error $\frac{1465}{330672}=0.00443037$.
\end{prop}

\begin{proof}
Let $\set{a_1, a_2, a_3, a_4}$ be an optimal set of four-means such that $a_1<a_2<a_3<a_4$. Since the optimal quantizers are the centroids of their own Voronoi regions, we have $0<a_1<a_2<a_3<a_4<1$. As in the previous two propositions, all the optimal quantizers can not lie in one side of the point $\frac {13}{20}$. Thus, $\frac 12<\frac{13}{20}\leq a_4$. We show that $0<a_1<\frac 12(a_1+a_2)< \frac 12<a_2<a_3<a_4<1$. If $0<a_1<\frac 12(a_1+a_2)< \frac 12<a_2<a_3<a_4<1$, then the distortion error is
\begin{align*}
&\int_0^{\frac{a_1+a_2}{2}} \frac{2}{5} (x-a_1)^2  dx+\int_{\frac{a_1+a_2}{2}}^{\frac{1}{2}} \frac{2}{5}  (x-a_2)^2 dx+\int_{\frac{1}{2}}^{\frac{a_2+a_3}{2}} \frac{8}{5} (x-a_2)^2 dx+\int_{\frac{a_2+a_3}{2}}^{\frac{a_3+a_4}{2}} \frac{8}{5} (x-a_3)^2  dx\\
&\qquad +\int_{\frac{a_3+a_4}{2}}^1 \frac{8}{5} (x-a_4)^2  dx\\
&=\frac{a_1^3}{10}+\frac{a_1^2 a_2}{10}-\frac{a_1 a_2^2}{10}+\frac{3 a_2^3}{10}+\frac{2 a_2^2 a_3}{5}-\frac{3 a_2^2}{5}-\frac{2 a_2 a_3^2}{5}+\frac{3 a_2}{10}+\frac{2 a_3^2 a_4}{5}-\frac{2 a_3 a_4^2}{5}-\frac{2 a_4^3}{5}\\
&\qquad +\frac{8 a_4^2}{5}-\frac{8 a_4}{5}+\frac{29}{60}.
\end{align*}
Now, obtain the partial derivatives of the above expression with respect to $a_1$, $a_2$, $a_3$, and $a_4$, and equating them to zero, we obtain a system of four equations. Solving the system of four equations, we obtain
$\set{a=\frac{59}{332},  b= \frac{177}{332},  c= \frac{239}{332},  d= \frac{301}{332}}$ with distortion error $\frac{1465}{330672}=0.00443037$. Since $V_4$ is the quantization error for four-means, we have $V_4\leq 0.00443037$. If $a_3\leq \frac 12$, then
\[V_4\geq \int_{\frac 12}^1 \min_{a\in \set{\frac 12, a_4}}\frac 85(x-a)^2 dx=-\frac{2 a_4^3}{5}+\frac{7 a_4^2}{5}-\frac{3 a_4}{2}+\frac{31}{60},\]
which is minimum when $a_4=\frac 56$, and the minimum value is $\frac{1}{135}=0.00740741>V_4$ which yields a contradiction. So, we can assume that $\frac 12<a_3$. We now show that $\frac 12 <a_2$. For the sake of contradiction, assume that $a_2\leq \frac 12$. Then, the following two cases can arise:

Case~1. $\frac 12 (a_2+a_3)\leq \frac 12<a_3$.

Then, the distortion error is given by
\begin{align*}
&\int_0^{\frac{a_1+a_2}{2}} \frac{2}{5}  (x-a_1)^2 dx+\int_{\frac{a_1+a_2}{2}}^{\frac{a_2+a_3}{2}} \frac{2}{5}  (x-a_2)^2  dx+\int_{\frac{a_2+a_3}{2}}^{\frac{1}{2}} \frac{2}{5}  (x-a_3)^2 dx+\int_{\frac{1}{2}}^{\frac{a_3+a_4}{2}} \frac{8}{5} (x-a_3)^2 dx\\
& \qquad +\int_{\frac{a_3+a_4}{2}}^1 \frac{8}{5} (x-a_4)^2  dx\\
&=\frac{a_1^3}{10}+\frac{a_1^2 a_2}{10}-\frac{a_1 a_2^2}{10}+\frac{a_2^2 a_3}{10}-\frac{a_2 a_3^2}{10}+\frac{3 a_3^3}{10}+\frac{2 a_3^2 a_4}{5}-\frac{3 a_3^2}{5}-\frac{2 a_3 a_4^2}{5}+\frac{3 a_3}{10}-\frac{2 a_4^3}{5}+\frac{8 a_4^2}{5}-\frac{8 a_4}{5}\\
& \qquad +\frac{29}{60}.
\end{align*}
Now, obtain the partial derivatives of the above expression with respect to $a_1$, $a_2$, $a_3$, and $a_4$, and equating them to zero, we obtain a system of four equations. Solving the system of four equations, we have
$\set{a_1= \frac{1}{8}, a_2= \frac{3}{8}, a_3= \frac{5}{8}, a_4=\frac{7}{8}}$ with distortion error $\frac{1}{192}=0.00520833>V_4$ which leads to a contradiction.

Case 2. $a_2\leq \frac 12\leq \frac 12(a_2+a_3)$.

Then, the distortion error is given by
\begin{align*}
&\int_0^{\frac{a_1+a_2}{2}} \frac{2}{5} (x-a_1)^2 \, dx+\int_{\frac{a_1+a_2}{2}}^{\frac{1}{2}} \frac{2}{5} (x-a_2)^2 \, dx+\int_{\frac{1}{2}}^{\frac{a_2+a_3}{2}} \frac{8}{5} (x-a_2)^2 \, dx+\int_{\frac{a_2+a_3}{2}}^{\frac{a_3+a_4}{2}} \frac{8}{5} (x-a_3)^2 \, dx\\
& \qquad +\int_{\frac{a_3+a_4}{2}}^1 \frac{8}{5} (x-a_4)^2 \, dx\\
&=\frac{a_1^3}{10}+\frac{a_1^2 a_2}{10}-\frac{a_1 a_2^2}{10}+\frac{3 a_2^3}{10}+\frac{2 a_2^2 a_3}{5}-\frac{3 a_2^2}{5}-\frac{2 a_2 a_3^2}{5}+\frac{3 a_2}{10}+\frac{2 a_3^2 a_4}{5}-\frac{2 a_3 a_4^2}{5}-\frac{2 a_4^3}{5}\\
& \qquad +\frac{8 a_4^2}{5}-\frac{8 a_4}{5}+\frac{29}{60}
\end{align*}
which is minimum when $a_1= \frac{1}{6}$, $a_2=\frac{1}{2}$, $a_3=\frac{7}{10}$,  $a_4= \frac{9}{10}$, and the minimum value is $\frac{61}{13500}=0.00451852>V_4$, which is a contradiction.

By Case~1 and Case~2, we can assume that $\frac 12<a_2$. If $\frac {21}{64}\leq a_1$, then
\[V_4\geq \int_0^{\frac {21}{64}} \frac{2}{5} (x-\frac {21}{64})^2 \, dx=\frac{3087}{655360}=0.00471039>V_4,\]
which leads to a contradiction, and so $a_1<\frac {21}{64}$. We now show that $\frac 12(a_1+a_2)<\frac 12$. For the sake of contradiction, assume that
$\frac 12\leq \frac 12(a_1+a_2)$. Then, $a_2\geq 1-a_1>1-\frac {21}{64}=\frac{43}{64}$, and so
\begin{align*}
V_4&\geq \int_0^{\frac{1}{2}} \frac{2}{5} (x-a_1)^2 \, dx+\int_{\frac{1}{2}}^{\frac{a_1+a_2}{2}} \frac{8}{5} (x-a_2)^2 \, dx+\int_{\frac{a_1+a_2}{2}}^{a_2} \frac{8}{5} (x-a_2)^2 \, dx\\
&=\frac{7 a_1^3}{15}+\frac{a_1^2 a_2}{5}-\frac{3 a_1^2}{5}-\frac{a_1 a_2^2}{5}-\frac{1}{15} (a_1-a_2)^3+\frac{3 a_1}{10}+\frac{a_2^3}{15}-\frac{1}{20},
\end{align*}
which is minimum when $a_1=\frac {21}{64}$ and $a_2=\frac{43}{64}$ ($P$-almost surely), and the minimum value is $\frac{3979}{491520}=0.0080953>V_4$, which gives a contradiction. Hence, we can assume that $\frac 12(a_1+a_2)<\frac 12$. Therefore, we can conclude that $0<a_1<\frac 12(a_1+a_2)< \frac 12<a_2<a_3<a_4<1$ which gives the optimal set of four-means as $\set{\frac{59}{332}, \frac{177}{332},   \frac{239}{332},  \frac{301}{332}}$, and the corresponding quantization error is $\frac{1465}{330672}=0.00443037$. Thus, the proof of the proposition is complete.
\end{proof}

\begin{remark} Proceeding similarly as the proof of Proposition~\ref{prop45}, though cumbersome, one can obtain the optimal sets of $n$-means for the probability distribution $P$ with density function $f(x)$ for $n=5, 6, 7$, etc. The problem of two uniform regions of different densities has two parameters, the ratio of the densities and the ratio of the lengths. The general formula to obtain an optimal set of $n$-means, for any $n\in \D N$, in this direction is still not known.
\end{remark}

\nd\tbf{Acknowledgement:} The author is grateful to Professor Carl P. Dettmann of University of Bristol for helpful discussions.

\end{document}